\documentclass[11pt]{amsart}
\usepackage[utf8]{inputenc}
\usepackage[T1]{fontenc}
\usepackage{amsmath,amssymb,amsfonts,textcomp,amsthm,xifthen,graphicx,color,pgfplots}
\usepackage{enumerate}
\usepackage{enumitem}
\usepackage{pgfplotstable}
\usepgfplotslibrary{groupplots}
\usepgfplotslibrary{colorbrewer}
\pgfplotsset{compat=newest}
\usepgfplotslibrary{external}
\usepackage{fullpage}
\usepackage[pdftex,
  pdfauthor={Thomas F\"{u}hrer and Sergio Rojas},
            pdftitle={A posteriori analysis of neural network approximations},
            ]{hyperref}

\usepackage{stmaryrd}
\usepackage[ruled,linesnumbered]{algorithm2e}
\newtheorem{theorem}{Theorem}

\newtheorem{lemma}[theorem]{Lemma}
\newtheorem{corollary}[theorem]{Corollary}

\newtheorem{remark}[theorem]{Remark}
\newcommand\net{\mathcal{N}}

\newcommand{\R}{\mathbb{R}}
\newcommand{\N}{\mathbb{N}}

\newcommand{\osc}{\mathrm{osc}}
\newcommand{\loss}{\mathcal{L}}

\newcommand{\RT}[1]{\boldsymbol{\mathrm{R\!T}}^{#1}}

\newcommand{\Patch}{\Omega}

\newcommand{\diam}{\mathrm{diam}}
\DeclareMathOperator*{\argmin}{arg\, min}

\newcommand{\bbeta}{\boldsymbol{\beta}}
\newcommand{\bA}{\boldsymbol{A}}

\newcommand{\bR}{\boldsymbol{R}}
\newcommand{\bP}{\boldsymbol{P}}

\newcommand{\bx}{\boldsymbol{x}}
\newcommand{\br}{\boldsymbol{r}}

\newcommand{\btau}{\boldsymbol{\tau}}
\newcommand{\bsigma}{\boldsymbol{\sigma}}

\newcommand{\jump}[1]{\llbracket{#1}\rrbracket}

\newcommand{\set}[2]{\big\{#1\,:\,#2\big\}}

\newcommand{\ip}[2]{(#1\hspace*{.5mm},#2)}
\newcommand{\dual}[2]{\langle#1\hspace*{.5mm},#2\rangle}

\renewcommand{\div}{\operatorname{div}}

\newcommand{\Hdivset}[1]{\boldsymbol{H}(\div;#1)}

\newcommand{\normal}{{\boldsymbol{n}}}

\newcommand{\mesh}{\mathcal{T}}

\newcommand{\faces}{\mathcal{F}}

\newcommand{\tr}{\operatorname{tr}}

\newcommand{\trdiv}[1]{\tr_{#1}^\mathrm{div}}
\newcommand{\trgrad}[1]{\tr_{#1}^\nabla}
\newcommand{\ran}{\mathrm{ran}}

\begin{document}

% Title and Authors
%\title{A posteriori analysis of neural network approximations}
\title{A posteriori analysis of neural network approximations of partial differential equations}
\date{\today}

\author{Thomas F\"{u}hrer}
\address{Facultad de Matem\'{a}ticas, Pontificia Universidad Cat\'{o}lica de Chile, Santiago, Chile}
\email{thfuhrer@uc.cl}
\author{Sergio Rojas}
\address{School of Mathematics, Monash University, Melbourne, Australia}
\email{Sergio.Rojas@monash.edu}
\thanks{{\bf Acknowledgment.} 
This work was supported by ANID through FONDECYT projects 1250070 (TF) and 1240643 (SR), and by the National Center for Artificial Intelligence CENIA FB210017, Basal ANID. Parts of the presented research have been conducted while the first author was hosted by the School of Mathematics at Monash University, Melbourne, Australia.}


%%%%%%%%%%%%%%%%%%%%%%%%%%%%%%%%%%%%%%%%%%%%%%%%%%%%%%%%%%%%%%%%%%%%%
% ABSTRACT
\begin{abstract}
With the aim of defining robust losses for the neural network approximation of partial differential equations, we study, in a general setting, a posteriori estimates used in finite element analysis.
We show that the error is equivalent to the sum of two residuals provided that the underlying variational formulation is well-posed. 
The first contribution is the projection of the residual to a finite-dimensional space and is therefore computable, 
while the second one can be reliably estimated by a computable upper bound in many practical scenarios. 
Two important applications can be considered during the neural network optimization (training) process: first, the computable lower bound is used as loss functional and both contributions are used to monitor the error at each solver step; or second, the two estimators are included in the loss functional, thereby providing total control over the error. 
As a model problem, we consider a second-order elliptic partial differential equation and discuss several variational formulations, including different approaches for incorporating boundary conditions into the estimators. Various numerical experiments are presented to validate our findings.
\end{abstract}
%%%%%%%%%%%%%%%%%%%%%%%%%%%%%%%%%%%%%%%%%%%%%%%%%%%%%%%%%%%%%%%%%%%%%
% Make Title
\maketitle
%\tableofcontents % TOFU: Not necessary
%%%%%%%%%%%%%%%%%%%%%%%%%%%%%%%%%%%%%%%%%%%%%%%%%%%%%%%%%%%%%%%%%%%%%

%%%%%%%%%%%%%%%%%%%%%%%%%%%%%%%%%%%%%%%%%%%%%%%%%%%%%%%%%%%%%%%%%%%%%
\section{Introduction}
%%%%%%%%%%%%%%%%%%%%%%%%%%%%%%%%%%%%%%%%%%%%%%%%%%%%%%%%%%%%%%%%%%%%%
Neural network-based methods such as Physics-Informed Neural Networks (PINNs)~\cite{raissi2019physics} have emerged as a powerful alternative to traditional numerical methods for approximating solutions to partial differential equations (PDEs), particularly in data-scarce scenarios. A key advantage of these methods lies in their ability to handle parametric PDEs, by enabling a single surrogate model to approximate the solution map across the parameter space. This capability naturally motivates a rigorous study of their approximation properties for the resolution of PDEs in the absence of data, where the governing equations and constraints provide the sole source of supervision.

At the methodological level, the core idea behind PINNs is to embed the underlying physical laws directly into the training process. This is achieved by augmenting the standard loss functional with penalty terms that measure deviations (residuals) from the governing equations, thereby framing the learning task as a residual minimization problem for the neural network. Specifically, the network aims to approximate a function that satisfies the PDE in a strong or weak sense, up to a given set of interpolation constraints. 

Despite their appeal and widespread adoption, PINNs suffer from several limitations. On the one hand, the strong-form residual used in standard PINNs typically requires high regularity assumptions for the solution.
On the other hand, the collocation-based approximation of the residual norm, combined with improper weighting between different contributions (e.g., interior vs. boundary residuals), can lead to poor convergence, instability, or inaccurate predictions~\cite{krishnapriyan2021characterizing, wang2022and}.

To address these shortcomings, several alternative methodologies for defining loss functionals have recently been proposed, including 
Deep Least-Squares methods~\cite{cai2020deep}, Deep Ritz Method~\cite{yu2018deep}, Variational PINNs (VPINNs)~\cite{kharazmi2019variational, kharazmi2021hp}, Robust VPINNs (RVPINNs) \cite{RVPINN24}, Quasi-Optimal Least-Squares (QOLS)~\cite{MSS_QOLS24}, and Weak Adversarial Networks (WAN)~\cite{WAN}. Recently,~\cite{AinsworthDong25} introduced extended Galerkin neural networks, where singularity functions are added to a base neural network architecture.
Unlike the classical PINN formulation~\cite{raissi2019physics}, these approaches are based on variational arguments rather than point-wise residual minimization. From a theoretical perspective, variational formulations not only relax the regularity requirements on the solution but also enable the incorporation of classical tools from the finite element community. This has led to notable improvements in complex scenarios where collocation-based strategies struggle. Nevertheless, challenges remain, particularly regarding the numerical integration of loss terms, which cannot be computed exactly in general, and the correct balancing of the loss terms.

Another notable issue is that the neural network defines a manifold rather than a finite-dimensional space. This implies that, contrary to finite element-type formulations, the imposition of boundary conditions remains a subtle issue. Direct (strong) enforcement of boundary conditions within the network architecture not only limits the class of admissible domains but may also degrade performance in complex scenarios. In contrast, weak imposition strategies have shown more promise in such settings (see, e.g.,~\cite{RVPINN24}). Nevertheless, as pointed out in~\cite{berrone2023enforcing}, even variational enforcement techniques such as Nitsche’s method can lead to suboptimal solvers. Thus, in the authors’ view, it remains an open task to enforce boundary conditions while maintaining a proper balance among the loss components. 

\noindent\textbf{Contributions.} 
Motivated by the above, we propose a general \emph{a posteriori} error estimation framework aimed at improving error control during neural network training. The proposed approach builds on techniques from minimum residual methods (MINRES) and the derivation of a posteriori error estimators for FEM.

Considering $u\in U$ to be the solution of a well-posed PDE, 
\begin{align*}
  Bu = F, 
\end{align*}
where $U,V$ are Hilbert spaces, $F\in V'$ represents given volume forces and boundary data, and $B\colon U\to V'$ is an isomorphism which corresponds to a variational formulation of the PDE operator. 
Denoting with $u_\theta\in U$ the output of a neural network with trainable parameters represented by $\theta$, the error between the exact solution and its approximation $u_\theta$ is equivalent to the residual, 
\begin{align*}
  \|u-u_\theta\|_U \eqsim \|F-B u_\theta\|_{V'}.
\end{align*}
Here, the notation $A\eqsim B$ for $A,B\geq 0$ means that there exists a generic constant $C>0$ such that $C^{-1} A\leq B \leq C A$. If only one of the latter estimates holds true we write $A\lesssim B$ or $B\lesssim A$.
Therefore, minimizing the right-hand side also reduces the error.
A huge drawback in most scenarios is that the right-hand side, due to being a dual norm, is not computable in practice. 
In practice, one can replace $V$ by a finite-dimensional subspace $V_h$, for instance, a conforming Finite Element space, and consider $\eta^2(u_\theta) := \|F-B u_\theta\|^2_{V_h'}$ as the loss functional (see, e.g.,  \cite{taylor2023deep, RVPINN24} in the context of VPINNs). This choice is particularly effective during the early stages of training, provided that the test space $V_h$ is sufficiently rich, so that $\eta(u_\theta)$ captures the dominant contribution to the overall approximation error.
However, in general, this only gives a guaranteed lower bound
\begin{align*}
  \eta(u_\theta) \lesssim \|u-u_\theta\|_U,
\end{align*}
so that by minimizing $\eta(u_\theta)$, this does not necessarily imply reduction for the real error $\|u-u_\theta\|_U$.
In fact, it might happen that $\eta(u_\theta)\to 0$, whereas $\|u-u_\theta\|_U$ stagnates, reflecting the need for a further \emph{a posteriori} error estimation to reliably assess the true approximation error and identify when the training loss ceases to provide a robust indicator of $\|u-u_\theta\|_U$, thereby informing appropriate stopping or adaptation strategies during training.

The main objective of the work at hand is to derive a \emph{computable} upper bound for the, in general, not computable dual norm $\|F-Bu_\theta\|_{V'}$. For various variational formulations of second-order elliptic PDEs, we show that this upper bound can be decomposed into two contributions, 
\begin{align*}
  \eta(u_\theta) \lesssim \|u-u_\theta\|_U \lesssim \eta(u_\theta) + \rho(u_\theta),
\end{align*}
where $\rho(u_\theta)$ is a localized PDE residual derived by using techniques from the analysis of finite element a posteriori estimators. In particular, we also propose to use $\eta^2(u_\theta)+\rho^2(u_\theta)$ as loss functional. This guarantees reliable control of the exact error during the whole training.

% $$\textit{error} \lesssim \textit{MINRES error} + \textit{FEM error estimator},$$
% where the \emph{MINRES error} (denoted $\eta(u_\theta)$), which also serves as a computable lower bound, captures the projection of the residual onto a discrete test space, while the \emph{FEM error estimator} (denoted $\rho(u_\theta)$) is also computable and derived using techniques from finite element analysis. 
% %Assuming sufficiently accurate quadrature and well-posedness of the variational formulation, the total error becomes equivalent, up to multiplicative constants, to the sum of these two contributions. 
% We also demonstrate that such a decomposition can be utilized to decouple interior and boundary residual contributions.

Our methodology builds on recent developments in residual minimization techniques in discrete dual norms, which is the core of several finite element methods available in the literature, see the overview article~\cite{DPG_ActaNumerica} and references therein.
Closely related are the discrete minus one inner products for least-squares finite element methods~\cite{MR1415797,MR1602783}, which combine a discrete dual norm (realized by a preconditioner) with a mesh-size weighted $L^2$ term. Our estimator $\eta+\rho$ for the weak formulation is of the same type (see Remark~\ref{rem:BLP}).
We propose to incorporate boundary conditions weakly, relying on residual-minimization techniques for fractional norms~\cite{MSS24} and volume norms~\cite{MSS_QOLS24}.

Comparing our method with the existing literature, in \cite{cai2020deep}, the authors define the loss functional as an upper bound for a residual-based quantity, in the spirit of least-squares formulations. 
In~\cite{ERU25}, the authors provide a framework for a posteriori estimation of neural network outputs by restricting and extending residuals to simpler geometries. 
A FEM-based a posteriori error estimator is proposed in~\cite{berrone2022solving} to monitor the error of a VPINN formulation for second-order elliptic problems during the training. 
In~\cite{RVPINN24} (see also \cite{uriarte2025optimizing, udomworarat2025neural}), the \emph{MINRES error} $\eta(u_\theta)$ is used directly as the loss functional, aligning with VPINNs~\cite{kharazmi2019variational} when the test space basis is orthonormal. 
The work in~\cite{taylor2023deep} constructs such an orthonormal basis, while~\cite{badia2024finite} projects the neural network onto a finite element space, enabling strong boundary enforcement and exact residual integration for the residual in terms of the projected neural network. However, using $\eta(u_\theta)$ alone may yield misleading results. As shown in~\cite{RVPINN24}, this term can converge to zero during optimization, although the true error stagnates, indicating that the complementary component becomes dominant. Therefore, neglecting the second term $\rho(u_\theta)$ in the error estimate can reduce both robustness and accuracy. 
In the recent work~\cite{MSS_QOLS24} quasi-optimal least-squares methods are introduced that are based on the use of neural networks for trial and test spaces.

In our framework, both contributions are explicitly considered and can be integrated into training objectives: either (i) as diagnostic tools to monitor solver accuracy, or (ii) as components of the loss functional for guaranteed error control. We illustrate the approach for different variational formulations of second-order elliptic PDEs, highlighting strategies for boundary condition imposition and providing numerical evidence of its effectiveness. 
We also point out and discuss differences between our approach and classical PINNs.

\noindent\textbf{Structure of the paper.}
The remainder of this article is structured as follows. In Section~\ref{sec:abstract}, we present a general \emph{a posteriori} error estimation approach and indicate how to efficiently compute the lower bound. We further discuss the role of the estimator contributions in the context of classical finite element and minimum residual (MINRES) methods. Section~\ref{sec:notation} introduces the notation and auxiliary results required for the subsequent analysis. In Section~\ref{sec:modelpoisson}, we formulate a diffusion–advection–reaction model problem and discuss three distinct variational formulations, including three different strategies for imposing Dirichlet-type boundary conditions in a weak sense. Section~\ref{sec:numerics} presents a series of numerical experiments that validate the effectiveness of the proposed methodology and demonstrate its practical advantages in different scenarios. Finally, conclusions and potential directions for future research are outlined in Section~\ref{sec:conclusions}.

%%%%%%%%%%%%%%%%%%%%%%%%%%%%%%%%%%%%%%%%%%%%%%%%%%%%%%%%%%%%%%%%%%%%%
\section{A general error estimation framework}\label{sec:abstract}
%%%%%%%%%%%%%%%%%%%%%%%%%%%%%%%%%%%%%%%%%%%%%%%%%%%%%%%%%%%%%%%%%%%%%
Let $U, V$ denote Hilbert spaces, and let $B\colon U \to V'$ be a linear and bounded operator. We say that $B$ is bounded below if there exists $c_B>0$, such that
\begin{align}\label{eq:Biso}
  c_B \|w\|_U \leq \|Bw\|_{V'}  \quad\forall w\in U.
\end{align}
Note that if $B$ is bounded below, then $B$ has closed range and is injective.

Let $V_h\subseteq V$ denote a finite-dimensional subspace, for instance, a mesh-dependent finite element subspace of $V$, equipped with a norm $\|\cdot\|_{V_h}$, induced by the inner product $\ip{\cdot}\cdot_{V_h}$, uniformly equivalent to $\|\cdot\|_V$. That is, there exists $C_\mathrm{equiv}>0$, independent of $V_h$, such that
\begin{align}\label{eq:normequivV}
  C_\mathrm{equiv}^{-1} \|v\|_V \leq \|v\|_{V_h} \leq C_\mathrm{equiv}\|v\|_V \quad\forall v\in V_h.
\end{align}
Clearly, $C_\mathrm{equiv}=1$ if $\|\cdot\|_{V_h}=\|\cdot\|_V$.
For a given $u\in U$, define for any $w\in U$
\begin{align}\label{eq:defeta}
 \eta(w)  := \eta(w;u,V_h) = \sup_{0\neq v\in V_h} \frac{b(u-w,v)}{\|v\|_{V_h}} = \|B(u-w)\|_{V_h'},
\end{align}
where $b\colon U\times V\to\R$ denotes the bilinear form induced by the operator $B$.
Finally, let $\Pi_h\colon V\to V_h$ denote a linear operator with $\|\Pi_h\|<\infty$, and set for any $w\in U$
\begin{align}\label{eq:defmu}
  \mu(w) := \mu(w;u,\Pi_h) = \sup_{0\neq v\in V} \frac{b(w-u,v-\Pi_hv)}{\|v\|_{V}}.
\end{align}
In our analysis, we also use the kernel space
\begin{align}\label{eq:defkernel}
  U_0(\Pi_h) = \set{w\in U}{b(w,v-\Pi_h v) = 0 \quad\forall v\in V}.
\end{align}
Note that $0\in U_0(\Pi_h)$ and $U_0(\Pi_h)$ is closed since it is the kernel of the linear and bounded operator $U\to V'$, $w\mapsto (v\mapsto b(w,v-\Pi_hv))$.
The following result will serve as the basis of our further investigations and discussions. 

\begin{theorem}\label{thm:equivest}
  Let $u\in U$ be given and $\eta$, $\mu$, $U_0(\Pi_h)$ be defined as in~\eqref{eq:defeta}--\eqref{eq:defkernel}.
  Then, for all $w\in  U$
  \begin{subequations}\label{eq:equivest}
  \begin{align}\label{eq:equivest:a}
    C_\mathrm{equiv}^{-1} \eta(w) \leq \|B(u-w)\|_{V'} \leq
    C_\mathrm{equiv} \|\Pi_h\| \eta(w) + \mu(w)
  \end{align}
  and
  \begin{align}\label{eq:equivest:b}
    \mu(w) \leq \|1-\Pi_h\|\inf_{w_0\in U_0(\Pi_h)}\|B(u-w-w_0)\|_{V'}
    \leq \|1-\Pi_h\|\|B(u-w)\|_{V'}.
  \end{align}
  \end{subequations}
\end{theorem}
\begin{proof}
  Let $w\in U$ be given. The lower bound in~\eqref{eq:equivest:a} follows directly from the norm equivalence~\eqref{eq:normequivV} and $V_h\subseteq V$.
  To obtain the upper bound, write $b(u-w,v) = b(u-w,\Pi_h v) + b(u-w,v-\Pi_h v)$. Then, the triangle inequality, boundedness of $\Pi_h$ and the estimates~\eqref{eq:normequivV} and~\eqref{eq:Biso}, show that
  \begin{align*}
    \|B(u-w)\|_{V'} & \leq \sup_{0\neq v\in V} \frac{b(u-w,\Pi_hv)}{\|v\|_V} 
    + \sup_{0\neq v\in V} \frac{b(u-w,v-\Pi_hv)}{\|v\|_V} \\
    &\leq \|\Pi_h\|\sup_{0\neq v\in V} \frac{b(u-w,\Pi_hv)}{\|\Pi_h v\|_V} + \mu(w)
    \\
    &\leq \|\Pi_h\| \sup_{0\neq v_h\in V_h} \frac{b(u-w,v_h)}{\|v_h\|_V} + \mu(w)
    \\
    &\leq C_\mathrm{equiv} \|\Pi_h\|\eta(w) + \mu(w).
  \end{align*}
  To prove the first inequality in~\eqref{eq:equivest:b}, let $w_0\in U_0(\Pi_h)$ be arbitrary. Observe that
  \begin{align*}
    |b(u-w,v-\Pi_h v)| = |b(u-w-w_0,v-\Pi_h v)| &\leq  \|B(u-w-w_0)\|_{V'}\|1-\Pi_h\|\|v\|_V.
  \end{align*}
  Dividing by $\|v\|_V$, taking the supremum over $V\setminus\{0\}$, and noting that $w_0\in U_0(\Pi_h)$ is arbitrary, proves the first inequality in~\eqref{eq:equivest:b}. The second inequality follows from recalling that $0\in U_0(\Pi_h)$, concluding the proof.
\end{proof}

\begin{remark}
  Since $B$ is bounded, it follows from~\eqref{eq:equivest:a} that 
  \begin{align*}
    C_\mathrm{equiv}^{-1} \eta(w)\leq \|B\| \|u-w\|_U
  \end{align*}
  and from~\eqref{eq:equivest:b} that
  \begin{align*}
    \mu(w) \leq \|B\|\,\|1-\Pi_h\|\inf_{w_0\in U_0(\Pi_h)} \|u-w-w_0\|_U \leq \|B\|\,\|1-\Pi_h\|\,\|u-w\|_U.
  \end{align*}
  Suppose that $B$ is bounded below~\eqref{eq:Biso}. Then,~\eqref{eq:equivest:a} implies that
  \begin{align*}
    c_B \|u-w\|_U \leq C_\mathrm{equiv} \|\Pi_h\| \eta(w) + \mu(w).
  \end{align*}
  In particular, 
  \begin{align*}
    \|u-w\|_U \eqsim \|B(u-w)\|_{V'} \eqsim \eta(w) + \mu(w) \quad\forall w\in U, 
  \end{align*} where the involved constants only depend on $c_B$, $\|B\|$, $C_\mathrm{equiv}$, $\|\Pi_h\|$.  
\end{remark}

\begin{remark}
  Let $U_h\subset U , V_h \subset V$ denote (conforming) finite element subspaces. If ``$w\in U$'' is replaced by ``$w\in U_h$'' in Theorem~\ref{thm:equivest} and assuming that $\Pi_h\colon V\to V_h$ satisfies the Fortin condition $b(w,v-\Pi_hv)=0$ for $w\in U_h$, $v\in V$, then Theorem~\ref{thm:equivest} reduces to a version of~\cite[Theorem 2.1]{DPGaposteriori} for bounded below operators $B$. In Section~\ref{sec:observations} below, we give further details on relations to (discontinuous) Petrov--Galerkin methods.
\end{remark}

For our applications in mind, we have that $V = \prod_{j=1}^n V_j$ where $V_j$, $j=1,\dots,n$ are Hilbert spaces and $B\colon U\to V'$ can be written as $B  = \sum_{j=1}^n B_j$, with $B_j\colon U\to V_j'$ (and the canonical embedding $V_j'\to V'$), $j=1,\dots,n$. 
This decomposition naturally arises in variational formulations of well-posed PDE problems of the form $Bu=F$ in $V'$, for instance, in mixed formulations or when separating interior residuals from boundary contributions.
Then,
\begin{align*}
  \|B(u-w)\|_{V'}^2 = \sum_{j=1}^n \|B_j(u-w)\|_{V_j'}^2.
\end{align*}
For all $j=1,\dots,n$ let $V_{h,j}\subset V_j$ denote a finite-dimensional subspace, $\Pi_{h,j}\colon V_j\to V_{h,j}$ a linear and bounded operator and define $\eta_j$, $\mu_j$ as in~\eqref{eq:defeta}--\eqref{eq:defmu} replacing $B$, $V$, $V_h$, $\Pi_h$ by $B_j$, $V_j$, $V_{h,j}$, $\Pi_{h,j}$. Similarly, let $\|\cdot\|_{V_{h,j}}$ denote an equivalent norm on $V_{h,j}$ with equivalence constant $C_\mathrm{equiv,j}$, $j=1,\dots,n$.
The next result follows from applying Theorem~\ref{thm:equivest} for all $j=1,\dots,n$.
\begin{corollary}\label{cor:equivest}
  Under the aforegoing assumptions, define
  \begin{align*}
    \eta(w)^2 = \sum_{j=1}^n \eta_j(w)^2, \qquad \mu(w)^2 = \sum_{j=1}^n \mu_j(w)^2.
  \end{align*}
  Then, 
  \begin{align*}
    C_1 \eta(w) \leq \|B(u-w)\|_{V'} \leq C_2\eta(w) + \sqrt{2}\mu(w) \quad\forall w\in U.
  \end{align*}
  Here, $C_1 = \min\{C_{\mathrm{equiv},j}^{-1}\,:\,j=1,\dots,n\}$, $C_2 = \sqrt{2}\max\{C_{\mathrm{equiv,j}}\|\Pi_{h,j}\|\,:\,j=1,\dots,n\}$.
  Furthermore, with $C_3 = \max\{\|1-\Pi_{h,j}\|\,:\,j=1,\dots,n\}$,
  \begin{align*}
    \mu(w) \leq C_3\|B(u-w)\|_{V'} \quad\forall w\in U.
  \end{align*}
  If $B$ is additionally bounded below, then
  \begin{align*}
    \|u-w\|_U \eqsim \eta(w) + \mu(w) \quad\forall w\in U
  \end{align*}
  where the involved constants only depend on $C_1$, $C_2$, $C_3$, $c_B$ and $\|B\|$.
\end{corollary}

\subsection{Efficient computation of the lower bound}\label{sec:computelowerbound}

In this short section, we review a general approach for computing $\eta(w)$. To that end, let $\ip\cdot\cdot_{V_h}$ denote the inner product on $V_h$ that induces the norm $\|\cdot\|_{V_h}$, and define, for each $w\in U$, $\phi_h(w) \in V_h$ as the discrete Riesz representative of the residual; that is, the discrete function satisfying
\begin{align}\label{eq:rieszrep:discrete}
  \ip{\phi_h(w)}{v}_{V_h} = F(v)-b(w,v) \quad\forall v\in V_h.
\end{align}
Then, 
\begin{align*}
  \eta(w) = \sup_{0\neq v\in V_h} \frac{b(u-w,v)}{\|v\|_{V_h}} 
  = \sup_{0\neq v\in V_h} \frac{F(v)-b(w,v)}{\|v\|_{V_h}}
  = \sup_{0\neq v\in V_h}  \frac{\ip{\phi_h(w)}{v}_{V_h}}{\|v\|_{V_h}}
  = \|\phi_h(w)\|_{V_h}.
\end{align*}
Let $v_1,\cdots,v_N$, with $N=\dim(V_h)$ denote a basis of $V_h$, and define the Riesz matrix $\bR\in\R^{N\times N}$ by
\begin{align*}
  \bR_{jk} = \ip{v_k}{v_j}_{V} \quad j,k=1,\ldots,N.
\end{align*}
Let $\bP\in\R^{N\times N}$ denote the Riesz matrix with respect to the $\ip\cdot\cdot_{V_h}$ inner product, i.e., 
\begin{align*}
  \bP_{jk} = \ip{v_k}{v_j}_{V_h} \quad j,k=1,\ldots,N.
\end{align*}
By the norm equivalence assumption~\eqref{eq:normequivV} we have that
\begin{align*}
  C_\mathrm{equiv}^{-2} \bx\cdot \bR\bx \leq \bx\cdot\bP\bx \leq C_\mathrm{equiv}^2 \bx\cdot \bR\bx
  \quad\forall \bx\in\R^N.
\end{align*}
This means that $\bR$ and $\bP$ are spectrally equivalent. In particular, their inverses $\bR^{-1}$ and $\bP^{-1}$ are symmetric positive definite matrices that are also spectrally equivalent.

To compute $\|\phi_h(w)\|_{V_h}$, it suffices to evaluate $\bP^{-1}$ on a single vector, 
\begin{align*}
  \|\phi_h(w)\|_{V_h}^2 = \br(w)\cdot\bP^{-1}\br(w), 
\end{align*}
where $\br(w)\in\R^N$ is the residual vector with entries $(\br(w))_j = F(v_j)-b(w,v_j)$, $j=1,\ldots,N$.
The computation of $\|\phi_h(w)\|_{V_h}$ is efficient if the action of $\bP^{-1}$ can be computed efficiently. Two relevant situations will be considered in this work:
First, when $\bR$ has a block-diagonal structure and each block can be efficiently inverted. In that case it suffices to set $\ip\cdot\cdot_{V_h} = \ip{\cdot}{\cdot}_V$, so that $\bP=\bR$. 
Second, when the direct computation of $\bR^{-1}$ is impractical, but it can be efficiently approximated by a preconditioner matrix $\bP^{-1}$. 

\subsection{Observations and some notes on relations to finite element methods}\label{sec:observations}
In what follows we discuss some observations implied by Theorem~\ref{thm:equivest}.
First, estimates~\eqref{eq:equivest} show the equivalence
\begin{align*}
  \|u-w\|_U \eqsim \eta(w) + \mu(w),
\end{align*}
provided that $B$ is bounded below, which we assume for the remainder of this section.
While $\eta(w)$ is computable, $\mu(w)$ is, in general, not, as it involves taking the supremum over the whole (infinite-dimensional) space $V$.

In this work, we propose using techniques from the a posteriori analysis of finite element methods to localize $\mu$ and obtain a computable upper bound $\rho$. That is, we aim to provide an estimate of the form
\begin{align}\label{eq:locest}%localization of estimator
  \mu(w) = \sup_{0\neq v\in V} \frac{b(u-w,v-\Pi_h v)}{\|v\|_V} \leq C \rho(w) \quad\forall w\in U
\end{align}
with some generic constant $C>0$. In Section~\ref{sec:modelpoisson}, we provide several examples of how to obtain such an estimate for a general second-order problem with uniformly bounded physical parameters. %and a singularly perturbed problem.
Note, however, that in general $\rho(w)$ is not a lower bound for the error $\|u-w\|_U$, i.e., the estimate
\begin{align*}
  \rho(w) \lesssim \|u-w\|_U \text{ does not hold in general.}
\end{align*}

Let us discuss the estimator contributions $\eta$, $\mu$, $\rho$ in the context of two popular numerical methods; the Galerkin method and the (discontinuous) Petrov--Galerkin method with optimal test functions or, more generally, the Minimum Residual method (MINRES).

\noindent\emph{Galerkin method.} 
Suppose that $V=U$ and that $b\colon U\times U\to \R$ is a bounded and coercive bilinear form. 
%Then, as a consequence of the Lax--Milgram lemma the operator $B$ is an isomorphism so that all assumptions are satisfied.
Let $U_h\subseteq U$ denote a closed subspace. Set $F=Bu$. The Galerkin method reads: Find $u_h\in U_h$, such that
\begin{align*}
  b(u_h,v_h) = F(v_h) \quad\forall v_h \in V_h:=U_h.
\end{align*}
Galerkin orthogonality, $b(u-u_h,v_h) = 0$ for all $v_h\in U_h$, implies that $\eta(u_h) = 0$, hence, 
\begin{align*}
  \|u-u_h\|_U \eqsim \mu(u_h) = \sup_{0\neq v\in U} \frac{F(v-\Pi_h v)-b(u_h,v-\Pi_h v)}{\|v\|_U}.
\end{align*}
Depending on the problem under consideration, the latter term can be localized assuming certain regularity of the involved data and choosing $\Pi_h$ to be some quasi-interpolator. 
As an example, we mention mesh-size weighted
residual-based error estimation for second-order elliptic problems, see, e.g.,~\cite{Verfuerth94}. 
\smallskip\\
\noindent\emph{Petrov--Galerkin method with optimal test functions.}
Suppose that $U_h\subset U$, $V_h\subset V$ are finite dimensional subspaces, set $F:=Bu$ and define
\begin{align*}
  u_h = \argmin_{w\in U_h} \|Bw-F\|_{V_h'} = \argmin_{w\in U_h} \eta(w).
\end{align*}
It is well known that the latter minimization problem admits a unique solution if there exists a Fortin operator $\Pi_h\colon V\to V_h$ satisfying $b(w,v-\Pi_h v) = 0$, for all $v\in V$ and $w\in U_h$, see~\cite{DPG_ActaNumerica}. Note that the existence of a Fortin operator also implies that
\begin{align*}
  \mu(w) &= \sup_{0\neq v\in V} \frac{F(v-\Pi_h v)-b(w,v-\Pi_h v)}{\|v\|_V} = \sup_{0\neq v\in V} \frac{F(v-\Pi_h v)}{\|v\|_V}
  \\
  &= \|F\circ(1-\Pi_h)\|_{V'} =: \osc(F).
\end{align*}
Therefore, $\mu(w)\neq 0$ in general, but it is independent of $w$. 
Finally, observe that
\begin{align*}
  \argmin_{w\in U_h} \eta(w)+\mu(w) = \argmin_{w\in U_h} \|Bw-F\|_{V_h'} + \mu(w) = \argmin_{w\in U_h} \|Bw-F\|_{V_h'} = u_h.
\end{align*}
Summarizing these observations, we always have
\begin{align*}
  \|u - u_h\|_U \eqsim \eta(u_h) + \mu(u_h).
\end{align*}
In the FEM context, for Galerkin methods we have $\eta(u_h) = 0$, so the error depends solely on $\mu(u_h)$, while for MINRES (Petrov--Galerkin) methods, $\mu(u_h)$ is independent of $u_h$, and the error is typically governed by $\eta(u_h)$. However, in the general case where $U_h$ is not a finite element space, for instance, when considering a nonlinear manifold induced by a neural network architecture $u_\theta$, both contributions $\eta(u_\theta)$ and $\mu(u_\theta)$ must be considered to ensure robust error estimations.
\section{Notation and auxiliary results}\label{sec:notation}
In this section, we introduce notations for Lebesgue and Sobolev spaces, meshes and other relevant definitions together with some auxiliary results that will be used later in the analysis of the model problem presented in Section~\ref{sec:modelpoisson}. In the following, let $\Omega\subset \R^d$, $d\geq 2$ denote a bounded connected Lipschitz domain with polytopal boundary $\Gamma:= \partial\Omega$.

\subsection{Volume and trace spaces}
For a fixed Banach space $X$, we denote by $L^p(\Omega;X)$, for $p\in[1,\infty]$, the Lebesgue space of functions with values in $X$. If $X=\R$, we simply write $L^p(\Omega)=L^p(\Omega;\R)$. 
The (standard) inner product on $L^2(\Omega)$ is denoted by $\ip{\cdot}\cdot_\Omega$, and the induced norm by $\|\cdot\|_\Omega$.
We also use the notation $\ip{\cdot}\cdot_\Omega$ and $\|\cdot\|_\Omega$ for the inner product and induced norm, respectively, in $L^2(\Omega;\R^d)$ and $L^2(\Omega;\R_\mathrm{sym}^{d\times d})$ where $\R_\mathrm{sym}^{d\times d}$ is the space of $\R$-valued symmetric matrices.
We use analogous notations for spaces defined over subdomains $\omega\subseteq\Omega$ or  $\omega\subseteq \Gamma$.

We denote by $H^1(\Omega)$ the Sobolev space of all $L^2(\Omega)$ functions whose first-order weak derivatives also belong to $L^2(\Omega)$. 
We equip this space with the norm
\begin{align*}
  \|u\|_{H^1(\Omega)} = \sqrt{\|\nabla u\|_\Omega^2 +\|u\|_\Omega^2} \quad\forall u\in H^1(\Omega).
\end{align*}
The space of all $L^2(\Omega;\R^d)$ functions whose weak divergence belongs to $L^2(\Omega)$ is denoted by $\Hdivset\Omega$ and we endow it with the norm
\begin{align*}
  \|\bsigma\|_{\Hdivset\Omega} = \sqrt{\|\div\bsigma\|_\Omega^2 + \|\bsigma\|_\Omega^2} \quad\forall \bsigma\in\Hdivset\Omega.
\end{align*}

Let us introduce the trace operators $\trgrad\Omega\colon H^1(\Omega) \to (\Hdivset\Omega)'$ and $\trdiv\Omega\colon \Hdivset\Omega\to (H^1(\Omega))'$, defined by
\begin{align*}
  \dual{\trgrad\Omega u}{\btau}_\Gamma &:= \ip{\div\btau}{u}_\Omega + \ip{\btau}{\nabla u}_\Omega, \\
  \dual{\trdiv\Omega \bsigma}{v}_\Gamma &:= \ip{\div\bsigma}v_\Omega + \ip{\bsigma}{\nabla v}_\Omega,
\end{align*}
for all $u,v\in H^1(\Omega)$, $\bsigma,\btau\in\Hdivset\Omega$. The corresponding trace spaces are defined as $H^{1/2}(\Gamma) := \ran(\trgrad\Omega)$, $H^{-1/2}(\Gamma) := \ran(\trdiv\Omega)$, and are endowed with the respective quotient norms
\begin{align*}
\|\widehat u\|_{1/2,\Gamma} &= \inf\set{\|u\|_{H^1(\Omega)}}{\trgrad\Omega u = \widehat u}, \\
\|\widehat \sigma\|_{-1/2,\Gamma} &= \inf\set{\|\bsigma\|_{\Hdivset\Omega}}{\trdiv\Omega \bsigma = \widehat \sigma}.
\end{align*}
Let us note that the trace spaces are in duality with each other. 
The following basic result provides a more explicit formulation of this fact. Its proof follows from the more general statement of~\cite[Lemma~2.2]{BreakingSpaces16}.
\begin{lemma}\label{lem:tracedualities}
  The following identities hold true for all $u\in H^1(\Omega)$, $\bsigma\in\Hdivset\Omega$:
  \begin{align*}
    \|\trgrad\Omega u\|_{1/2,\Gamma} &= \sup_{0\neq\widehat\tau\in H^{-1/2}(\Gamma)} \frac{\dual{\trgrad\Omega u}{\widehat\tau}_\Gamma}{\|\widehat\tau\|_{-1/2,\Gamma}} = \sup_{0\neq\btau \in \Hdivset\Omega} \frac{\dual{\trgrad\Omega u}{\btau}_\Gamma}{\|\btau\|_{\Hdivset\Omega}}, \\
    \|\trdiv\Omega \bsigma\|_{-1/2,\Gamma} &= \sup_{0\neq \widehat v \in H^{1/2}(\Gamma)} \frac{\dual{\trdiv\Omega\bsigma}{\widehat v}_\Gamma}{\|\widehat v\|_{1/2,\Gamma}} = \sup_{0\neq v \in H^1(\Omega)} \frac{\dual{\trdiv\Omega \bsigma}{v}_\Gamma}{\|v\|_{H^1(\Omega)}}.
  \end{align*}
\end{lemma}
Note that, by definition $\dual{\trgrad\Omega u}{\bsigma}_\Gamma = \dual{\trdiv\Omega\bsigma}u_\Gamma$. Therefore, for simplicity of notation, either term may be written as  $\dual{u}{\bsigma\cdot\normal}_\Gamma$ or $\dual{\bsigma\cdot\normal}u_\Gamma$ if the meaning is clear from the context. Here, $\normal$ denotes the exterior normal vector on $\Gamma$ pointing from $\Omega$ outwards.

We stress that $\trgrad\Omega$ can be identified with the classical trace operator $\gamma_0\colon H^1(\Omega) \to H^{1/2}(\Gamma)$ where $\gamma_0 u = u|_\Gamma$ for all $u$ in the dense subspace $C^\infty(\overline\Omega)$.
Furthermore, let $H^1(\Gamma)$ denote the space of $L^2(\Gamma)$ functions $u$ whose surface gradient $\nabla_\Gamma u$ is in $L^2(\Gamma;\R^d)$.

Finally, let $H_0^1(\Omega)$ denote the subspace of $H^1(\Omega)$ consisting of functions with vanishing trace. We endow this space with norm $\|\nabla\cdot\|_\Omega$, and note that, by Poincar\'e's inequality, 
\begin{align*}
  \|\nabla u\|_\Omega\eqsim \|u\|_{H^1(\Omega)} \quad\forall u\in H_0^1(\Omega). 
\end{align*}
The dual space $H^{-1}(\Omega) = (H_0^1(\Omega))'$ is equipped with the norm
\begin{align*}
  \|\phi\|_{H^{-1}(\Omega)} = \sup_{0\neq v\in H_0^1(\Omega)} \frac{\dual{\phi}{v}_\Omega}{\|\nabla v\|_\Omega} \quad\forall\phi\in H^{-1}(\Omega).
\end{align*}
We note that all duality brackets introduced are understood as continuously extended $L^2$ inner products.

\subsection{Meshes and mesh-dependent spaces}

\subsubsection{Volume mesh}
Let $\mesh$ denote a regular simplicial mesh of $\Omega$, and let $h_\mesh\in  L^\infty(\Omega)$ be the local mesh-size function defined by $h_\mesh|_T:=\diam(T)$ for each $T\in\mesh$.
The set of facets is denoted by $\faces_\mesh$, and $\faces_\mesh^0\subseteq\faces_\mesh$ denotes the set of interior facets. 
Let $P^k(T)$ denote the space of polynomials of degree $\leq k\in\N_0$ with domain $T\in\mesh$. The space of $\mesh$-piecewise polynomials is denoted by $P^k(\mesh)$. 
The $L^2(\Omega)$ projection onto $P^k(\mesh)$ is denoted by $\pi_\mesh^k$.
Let $J_{\mesh,0}^{k+1}\colon L^2(\Omega)\to P^{k+1}(\mesh)\cap H_0^1(\Omega)$ denote a Scott--Zhang-type projection operator~\cite{ScottZhang90} with
\begin{align*}
  \|J_{\mesh,0}^{k+1}v\|_T &\lesssim \|v\|_{\Omega(T)}, \quad \|\nabla J_{\mesh,0}^{k+1} v\|_T \lesssim \|\nabla v\|_{\Omega(T)}, \\
  \|(1-J_{\mesh,0}^{k+1})v\|_T &\lesssim h_T \|\nabla v\|_{\Omega(T)},
\end{align*}
where, for each $T\in\mesh$, the patch $\Omega(T):=\set{T'\in\mesh}{\overline{T}\cap\overline{T'}\neq\emptyset}\subseteq \Omega$ denotes the domain associated with the vicinity of $T$.
Combining these estimates together with the trace inequality $\|w\|_F^2\lesssim \|w\|_T(h_T^{-1}\|w\|_T+\|\nabla w\|_T)$ shows that 
\begin{align}\label{eq:sztraceest}
  \|v-J_{\mesh,0}^{k+1}v\|^2_F \lesssim \|v-J_{\mesh,0}^{k+1}v\|_{T}\|\nabla v\|_{\Omega(T)}
  \lesssim h_T\|\nabla v\|_{\Omega(T)}^2
\end{align}
for $F\in \faces_T$, with $\faces_T$ being the $d+1$ boundary facets of $T\in\mesh$.

We also work with the space of element bubble functions $P_b^k(\mesh) = \set{v}{v|_T = qb_T \text{ with }q\in P^k(T), T\in\mesh}$, where $b_T\in P^{d+1}(T)$ denotes the product of the $d+1$ barycentric coordinate functions with support on $T\in\mesh$. 
\begin{lemma}\label{lem:szbubble}
  Let $k,m\in\N_0$. 
  There exists $J_{\mesh,0,b}^{k+1,m}\colon L^2(\Omega) \to P^{k+1}(\mesh)\cap H_0^1(\Omega)+P_b^m(\mesh)$, such that
  \begin{align*}
    \|J_{\mesh,0,b}^{k+1,m}w\|_{T} &\lesssim \|w\|_{\Omega(T)}, \quad \|\nabla J_{\mesh,0,b}^{k+1,m}w\|_{T} \,\lesssim \|\nabla w\|_{\Omega(T)}, \\
    \|w-J_{\mesh,0,b}^{k+1,m}w\|_{T} + h_T\|\nabla (w-J_{\mesh,0,b}^{k+1,m}w)\|_{T}&\lesssim h_T\|\nabla w\|_{\Patch(T)}, \quad \forall T\in\mesh, \\
    \ip{q_m}{w-J_{\mesh,0,b}^{k+1,m}w}_\Omega &= 0, \quad\forall q_m \in P^m(\mesh).
  \end{align*}
\end{lemma}
\begin{proof}
  For a proof we refer to, e.g.,~\cite[Proposition~4.1]{MSS24}.
\end{proof}
Furthermore, let $\RT{p}(\mesh)$ denote the Raviart--Thomas space of order $p\in\N_0$.
\subsubsection{Boundary mesh}
Let $\faces$ denote a simplicial partition of $\Gamma$ with
(local) mesh-size function $h_\faces\in L^\infty(\Gamma)$. We assume that there exists a partition $\widetilde\mesh$ of $\Omega$, such that the restrictions of its elements to the boundary generate $\faces$. Note that, in particular, we might use the restrictions of the bulk mesh $\mesh$ to the boundary to define a partition $\faces$ of $\Gamma$. 

Let $J_\faces^{p+1}\colon L^2(\Gamma)\to P^{p+1}(\faces)\cap H^1(\Gamma)$ denote the Scott--Zhang-type operator constructed and analyzed in~\cite{AFKPP13}.
We stress that, for each $s\in[0,1]$, the operator $J_\faces^{p+1}$ is a bounded mapping from $H^s(\Gamma)$ to $H^s(\Gamma)$.
Further, it enjoys the following approximation property
\begin{align*}
  \|h_\faces^{-1/2}(w-J_\faces^{p+1}w)\|_{\Gamma} + \|w-J_\faces^{p+1}w\|_{1/2,\Gamma} \lesssim 
  \|h_\faces^{1/2}\nabla_\Gamma w\|_\Gamma, \quad\forall w\in H^1(\Gamma)
\end{align*}
as well as (\cite[Proposition~3.1]{AFKPP13})
\begin{align}\label{eq:propszbou}
  \|h_\faces^{1/2}\nabla_\Gamma(1-J_\faces^{p+1})w\|_\Gamma \eqsim \|h_\faces^{1/2}(1-\pi_\faces^p)\nabla_\Gamma w\|_\Gamma,
  \quad\forall w\in H^1(\Gamma).
\end{align}
\subsubsection{Mesh-dependent spaces}
Consider the broken space
\begin{align*}
  H^1(\mesh) &= \set{v\in L^2(\Omega)}{v|_T\in H^1(T), T\in\mesh},
  \end{align*}
equipped with the (squared) norm
\begin{align*}
  \|v\|_{H^1(\mesh)}^2 &= \|\nabla_\mesh v\|_\Omega^2 + \|v\|_\Omega^2 := \sum_{T\in\mesh} \|\nabla v\|_T^2 + \|v\|_T^2, \quad v\in H^1(\mesh).
\end{align*}
Define the trace operator $\trdiv{\mesh}\colon \Hdivset\Omega\to H^1(\mesh)'$ by
\begin{align*}
  \dual{\trdiv\mesh\btau}{v}_{\partial\mesh}:=\dual{v}{\btau\cdot\normal}_{\partial\mesh}:= \sum_{T\in\mesh} \ip{\div\btau}{v}_T + \ip{\btau}{\nabla_\mesh v}_\Omega, \quad\btau\in\Hdivset\Omega, v\in H^1(\mesh).
\end{align*}
Set $H^{-1/2}(\partial\mesh):= \ran(\trdiv\mesh)$, and equip this space with the canonical quotient norm, denoted by $\|\cdot\|_{-1/2,\partial\mesh}$. 
Similar to Lemma~\ref{lem:tracedualities}, the following trace duality result holds, see also~\cite[Lemma~2.2]{BreakingSpaces16} in the context of discontinuous Petrov--Galerkin methods,
\begin{align}\label{eq:traceduality:broken}
  \|\lambda\|_{-1/2,\partial\mesh} = \sup_{0\neq v\in H^1(\mesh)} \frac{\dual{\lambda}{v}_{\partial\mesh}}{\|v\|_{H^1(\mesh)}}, \quad\forall \lambda\in H^{-1/2}(\partial\mesh).
\end{align}

%%%%%%%%%%%%%%%%%%%%%%%%%%%%%%%%%%%%%%%%%%%%%%%%%%%%%%%%%%%%%%%%%%%%%
\section{A second-order elliptic problem}\label{sec:modelpoisson}
%%%%%%%%%%%%%%%%%%%%%%%%%%%%%%%%%%%%%%%%%%%%%%%%%%%%%%%%%%%%%%%%%%%%%
As a model problem, we consider the second-order elliptic problem
\begin{subequations}\label{eq:modelPDE}
\begin{alignat}{2}
  -\div\bA\nabla u + \bbeta\cdot\nabla u + c u &= f &\quad&\text{in }\Omega, \\
  u &= g &\quad&\text{on } \Gamma := \partial\Omega.
\end{alignat}
\end{subequations}
Here, $\bA \in L^\infty(\Omega;\R^{d\times d}_\mathrm{sym})$ with uniformly bounded below eigenvalues, $\bbeta\in L^\infty(\Omega;\R^d)$, $\div\bbeta\in L^\infty(\Omega)$, and $c\in L^\infty(\Omega)$ with $-\tfrac12\div\bbeta + c \geq 0$.
We stress that, under the above assumptions, the weak formulation
\begin{align}\label{eq:weakform}
\begin{split}
  \ip{\bA\nabla u}{\nabla v}_\Omega + \ip{\bbeta\cdot\nabla u}{v}_\Omega + \ip{cu}v_\Omega &= \ip{f}v_\Omega
  \quad\forall v\in H_0^1(\Omega), 
  \\
  u|_\Gamma &= g,
\end{split}
\end{align}
admits a unique solution $u\in H^1(\Omega)$, given $f\in H^{-1}(\Omega)$ and $g\in H^{1/2}(\Gamma)$. 
In order to be able to localize residual contributions, we further assume throughout this article that
\begin{align*}
  \nabla(\bA)_{jk}\in L^\infty(\Omega;\R^d), \, j,k=1,\ldots,d, 
  \quad
  f\in L^2(\Omega), g\in H^1(\Gamma).
\end{align*}
In what follows, we discuss three formulations of the latter PDE, the weak formulation~\eqref{eq:weakform} (Section~\ref{sec:weakform}), another one based on broken test spaces (Section~\ref{sec:brokenform}), and a strong formulation (Section~\ref{sec:strongform}).
For all cases, we can write $V = V_\Omega\times V_\Gamma$ so that the operator $B\colon U\to V'\equiv V_\Omega'\times V_\Gamma'$ can be written as the sum of two operators, say, $B = B_\Omega+B_\Gamma$ with $B_\Omega\colon U\to V_\Omega'$ and $B_\Gamma\colon U\to V_\Gamma'$. In a similar fashion, we decompose $F = F_\Omega + F_\Gamma$.
In particular, $V_\Gamma = H^{-1/2}(\Gamma)$ or $V_\Gamma = \Hdivset\Omega$ (both choices are equivalent by Lemma~\ref{lem:tracedualities}) and $B_\Gamma = \trgrad\Omega$ (see Section~\ref{sec:notation}).
We choose the finite-dimensional test space 
$V_h \subset V$ in the same fashion, considering mesh-dependent finite element subspaces.
In any case, we have the following equivalence
\begin{align*}
  \|u-w\|_U \eqsim \|F_\Omega -B_\Omega w\|_{V_\Omega'} + \|F_\Gamma - B_\Gamma w\|_{V_\Gamma'}.
\end{align*}
Note that, by our abstract framework from Section~\ref{sec:abstract} (Corollary~\ref{cor:equivest}), we can treat both terms on the right-hand side independently.
We start by discussing the boundary trace term in Section~\ref{sec:boucond}.
The volume residual will be handled in Sections~\ref{sec:weakform}--\ref{sec:strongform}.

\subsection{Boundary estimators}\label{sec:boucond}
Let $B_\Gamma = \trgrad\Omega$. In view of the observations from Section~\ref{sec:abstract}, the goal is to define estimators $\eta_\Gamma$ and $\rho_\Gamma$ such that
\begin{align}\label{eq:bouequiv}
  \eta_\Gamma(w) \lesssim \|B_\Gamma w-g\|_{V_\Gamma'} \lesssim \eta_\Gamma(w) + \rho_\Gamma(w).
\end{align}
In what follows, we discuss various alternatives for deriving such estimates. The first option is based upon using $\Hdivset\Omega$ to define the trace norm, and this representation was proposed and utilized in~\cite{MSS_QOLS24}.
The second one uses $H^{-1/2}(\Gamma)$ and its use in MINRES methods was analyzed in~\cite{MSS24}.
To simplify notation, we write $w$ or $w|_\Gamma$ instead of $B_\Gamma w$. The last alternative uses an intrinsic definition of the $H^{1/2}(\Gamma)$ norm.

\subsubsection{First option}\label{sec:boucond:opt1}
For the first alternative we set $V_\Gamma = \Hdivset\Omega$ and recall from Lemma~\ref{lem:tracedualities} that
\begin{align*}
  \|\cdot\|_{1/2,\Gamma} = \sup_{0\neq \btau\in\Hdivset\Omega} \frac{\dual{\cdot}\btau_\Gamma}{\|\btau\|_{\Hdivset\Omega}}.
\end{align*}
We choose $V_{\Gamma,\widetilde\mesh} = \RT{p}(\widetilde\mesh) \subset V_\Gamma$ as the finite dimensional subspace, and recall that the boundary mesh $\faces$ is defined as the restriction of the volume mesh $\widetilde\mesh$ to the boundary. For all $w\in H^1(\Omega)$ with $w|_\Gamma\in H^1(\Gamma)$ set
\begin{subequations}\label{eq:defetarhobou1}
\begin{align}
  \eta_\Gamma(w) &= \sup_{0\neq \btau\in V_{\Gamma,\widetilde\mesh}} \frac{\dual{w-g}\btau_\Gamma}{\|\btau\|_{\Hdivset\Omega}}, \\
  \rho_\Gamma(w) &= \begin{cases}
    \|h_{\faces}^{1/2}\nabla_\Gamma(w-g)\|_\Gamma & p=0, \\
    \|h_{\faces}^{1/2}(1-\pi_\faces^{p-1})\nabla_\Gamma(w-g)\|_\Gamma & p>0.
  \end{cases}
\end{align}
\end{subequations}

\begin{lemma}\label{lem:bouest1}
  Suppose that $\eta_\Gamma$ and $\rho_\Gamma$ are defined as in~\eqref{eq:defetarhobou1}. Then, inequalities~\eqref{eq:bouequiv} are satisfied for all $g\in H^1(\Gamma)$ and $w\in H^1(\Omega)$ with $w|_\Gamma\in H^1(\Gamma)$.
\end{lemma}
\begin{proof}
  In view of the previously established abstract results, we define a bounded linear operator $\Pi_{\widetilde\mesh}\colon V_\Gamma \to V_{\Gamma,\widetilde\mesh}$.
  First, we consider $p>0$. 
  Note that $J_{\faces}^p$ is bounded in $H^{1/2}(\Gamma)$ and its dual $(J_{\faces}^p)'$ is therefore bounded in $H^{-1/2}(\Gamma)$. One also verifies that the range of the dual operator is in $P^p(\faces)$.
  Since $\Omega$ is a connected Lipschitz polytopal domain, there exists a discrete extension operator $E_{\widetilde\mesh}\colon P^p(\faces)\to\RT{p}(\widetilde\mesh)$ with
  \begin{align*}
    \trdiv\Omega E_{\widetilde\mesh}\lambda = \lambda \quad\text{and}\quad
    \|E_{\widetilde\mesh}\lambda\|_{\Hdivset\Omega} \lesssim \|\lambda\|_{-1/2,\Gamma} \quad\forall\lambda\in P^p(\faces),
  \end{align*}
  where the involved constant only depends on $\Omega$ and the shape-regularity of $\widetilde\mesh$, see~\cite{MR3522965}.
  Let $\btau\in \Hdivset\Omega$ be given and set $\Pi_{\widetilde\mesh} \btau:= E_{\widetilde\mesh}(J_{\faces}^p)'\trdiv\Omega\btau$.
  By construction, $\trdiv\Omega\Pi_{\widetilde\mesh}\btau = (J_{\faces}^p)'\trdiv\Omega\btau$. Moreover,
  \begin{align*}
    \|\Pi_{\widetilde\mesh} \btau\|_{\Hdivset\Omega} \lesssim \|(J_{\faces}^p)'\trdiv\Omega\btau\|_{-1/2,\Gamma} \lesssim \|\trdiv\Omega\btau\|_{-1/2,\Gamma} \leq \|\btau\|_{\Hdivset\Omega},
  \end{align*}
  which means that $\Pi_{\widetilde\mesh}$ is a bounded operator. 
  Further, 
  \begin{align*}
    |\dual{w-g}{\trdiv\Omega(\btau-\Pi_{\widetilde\mesh}\btau)}_\Gamma| &= |\dual{w-g}{(1-(J_{\faces}^p)')\trdiv\Omega\btau}_\Gamma|
    = |\dual{(1-J_{\faces}^p)(w-g)}{\trdiv\Omega\btau}_\Gamma|\\
    &\lesssim \|(1-J_{\faces}^p)(w-g)\|_{1/2,\Gamma}\|\trdiv\Omega\btau\|_{-1/2,\Gamma} \\
    &\lesssim \|h_{\faces}^{1/2}\nabla_\Gamma(1-J_{\faces}^p)(w-g)\|_\Gamma \|\btau\|_{\Hdivset\Omega}.
  \end{align*}
  The last estimates and~\eqref{eq:propszbou} show that
  \begin{align*}
  \sup_{0\neq \btau\in\Hdivset\Omega} \frac{\dual{w-g}{\trdiv\Omega(\btau-\Pi_{\widetilde\mesh}\btau)}_\Gamma}{\|\btau\|_{\Hdivset\Omega}}
  &\lesssim \|h_{\faces}^{1/2}\nabla_\Gamma(1-J_{\faces}^p)(w-g)\|_\Gamma
    \\
    &\eqsim \|h_\faces^{1/2}(1-\pi_\faces^{p-1})\nabla_\Gamma(w-g)\|_\Gamma
    = \rho_\Gamma(w),
  \end{align*}
  which completes the proof for $p>0$. The case $p=0$ is slightly more technical, but follows similar arguments.
  Instead of the operator $J_{\faces}^p$, we consider $J_{\faces}^1\pi_{\faces}^0$. We stress that one can show that this operator is bounded in $H^{1/2}(\Gamma)$, and its dual has range in $P^0(\faces)$. Moreover,
  \begin{align*}
    \|(1-J_{\faces}^1\pi_{\faces}^0)(w-g)\|_{1/2,\Gamma} \lesssim \|h_{\faces}^{1/2}\nabla_\Gamma(w-g)\|_\Gamma.
  \end{align*}
  Following the same arguments as in the case $p>0$ concludes the proof.
\end{proof}

\begin{remark}
  The efficient computation of $\eta_\Gamma(w)$ requires a spectrally equivalent preconditioner (Section~\ref{sec:computelowerbound}). 
  Efficient preconditioners for $\Hdivset\Omega$ are known in the literature (see, e.g.,~\cite{Hiptmair97}).
\end{remark}

\subsubsection{Second option}\label{sec:boucond:opt2}
For the second alternative we set $V_\Gamma = H^{-1/2}(\Gamma)$ and recall from Lemma~\ref{lem:tracedualities} that
\begin{align*}
  \|\cdot\|_{1/2,\Gamma} = \sup_{0\neq \psi\in  H^{-1/2}(\Gamma)} \frac{\dual{\cdot}{\psi}_\Gamma}{\|\psi\|_{-1/2,\Gamma}}.
\end{align*}
We choose $V_{\Gamma,\faces} = P^p(\faces) \subset V_\Gamma$ as the finite dimensional subspace, and set for all $w\in H^1(\Omega)$ with $w|_\Gamma\in H^1(\Gamma)$
\begin{subequations}\label{eq:defetarhobou2}
\begin{align}
  \eta_\Gamma(w) &= \sup_{0\neq \psi\in V_{\Gamma,\faces}} \frac{\dual{w-g}{\psi}_\Gamma}{\|\psi\|_{-1/2,\Gamma}}, \\
  \rho_\Gamma(w) &= \begin{cases}
    \|h_{\faces}^{1/2}\nabla_\Gamma(w-g)\|_\Gamma & p=0, \\
    \|h_{\faces}^{1/2}(1-\pi_\faces^{p-1})\nabla_\Gamma(w-g)\|_\Gamma & p>0.
  \end{cases}
\end{align}
\end{subequations}
Note that the definitions~\eqref{eq:defetarhobou2} are similar to~\eqref{eq:defetarhobou1}. The difference is that here we only need to consider discretizations of spaces defined over the manifold $\Gamma$ instead of $\Omega$, though at the price of dealing with a fractional norm.

\begin{lemma}\label{lem:bouest2}
  Suppose that $\eta_\Gamma$ and $\rho_\Gamma$ are defined as in~\eqref{eq:defetarhobou2}. Then, inequalities~\eqref{eq:bouequiv} are satisfied for all $g\in H^1(\Gamma)$ and $w\in H^1(\Omega)$ with $w|_\Gamma\in H^1(\Gamma)$.
\end{lemma}
\begin{proof}
  In the proof of Lemma~\ref{lem:bouest1} the operator $\Pi_{\widetilde\mesh}$ is constructed. By taking the normal trace of this operator to define $\Pi_{\Gamma,\faces}\colon V_\Gamma\to V_{\Gamma,\faces}$, we can follow the same lines of arguments.
  We omit further details. 
\end{proof}

The evaluation of $\eta_\Gamma(w)$ requires the evaluation of the $H^{-1/2}(\Gamma)$ inner product on $V_{\Gamma,\faces} = P^p(\faces)$. 
We note that (see, e.g.,~\cite[p.215]{MultilevelNorms})
%\begin{lemma}
  \begin{align*}
  \|\phi\|_{-1/2,\Gamma} \eqsim \|\pi_\faces^0\phi\|_{-1/2,\Gamma} + \|h_\faces^{1/2}(1-\pi_\faces^0)\phi\|_\Gamma 
  \quad\forall \phi \in P^p(\faces).
\end{align*}
%\end{lemma}
%
Based on the aforementioned equivalence, it suffices to focus on the fast evaluation of $ \|\pi_\faces^0 \phi\|_{-1/2, \Gamma}$ to compute $\eta_\Gamma$ efficiently.  
Optimal preconditioners for piecewise constant functions in negative Sobolev spaces include multilevel preconditioners (see, e.g.,~\cite{MultilevelNorms,STBook21}), as well as operator preconditioners (see, e.g.,~\cite{MR4044445}).
\subsubsection{Third option}\label{sec:boucond:opt3} 
As a third alternative, instead of working with definitions given by dual norms, we may directly work with the interpolation norm on trace space $H^{1/2}(\Gamma)$. 
Define for $p\in\N_0$
\begin{subequations}\label{eq:defetarhobou3}
\begin{align}
  \eta_\Gamma(w) &= \|J_\faces^{p+1}(w-g)\|_{1/2,\Gamma} \\
  \rho_\Gamma(w) &= \|h_\faces^{1/2}(1-\pi_\faces^p)\nabla_\Gamma(w-g)\|_\Gamma
\end{align}
\end{subequations}
for all $w\in H^1(\Omega)$ with $w|_\Gamma\in H^1(\Gamma)$.

\begin{lemma}\label{lem:bouest3}
  Suppose that $\eta_\Gamma$ and $\rho_\Gamma$ are defined as in~\eqref{eq:defetarhobou3}. Then,
  inequalities~\eqref{eq:bouequiv} are satisfied
  for all $g\in H^1(\Gamma)$ and $w\in H^1(\Omega)$, with $w|_\Gamma\in H^1(\Gamma)$.
\end{lemma}
\begin{proof}
  This simply follows from the triangle inequality and properties of the operator $J_\faces^{p+1}$ (see Section~\ref{sec:notation}).
\end{proof}

For the fast evaluation of $\eta_\Gamma$, we note that efficient preconditioners in $H^{1/2}(\Gamma)$ are well established (see, e.g.,~\cite{STBook21,MR4094780}).
\begin{remark}\label{rem:investbou}
  Using an inverse estimate and boundedness of $J_\faces^{p+1}$ one gets that
  \begin{align*}
    \eta_\Gamma(w) &= \|J_\faces^{p+1}(w-g)\|_{1/2,\Gamma} \lesssim \|h_\faces^{-1/2}J_\faces^{p+1}(w-g)\|_\Gamma
    \lesssim \|h_\faces^{-1/2}(w-g)\|_\Gamma.
  \end{align*}
  Replacing $\eta_\Gamma$ by the weighted $L^2(\Gamma)$ norm $\|h_\faces^{-1/2}(w-g)\|_\Gamma$ is a common technique in least-squares (based) finite element methods~\cite{BG09} to avoid dealing with non-local fractional norms.
  We emphasize that the weighted norm is not equivalent to $\eta_\Gamma(w)$. 
  However, it can be used as an upper bound in combination with $\rho_\Gamma(w)$ since
  \begin{align*}
    \|w-g\|_{1/2,\Gamma} \lesssim \eta_\Gamma(w) + \rho_\Gamma(w) \lesssim \|h_\faces^{-1/2}(w-g)\|_\Gamma +\|h_\faces^{1/2}(1-\pi_\faces^p)\nabla_\Gamma(w-g)\|_\Gamma.
  \end{align*}
\end{remark}

\subsection{Weak formulation}\label{sec:weakform}
We consider the weak formulation~\eqref{eq:weakform} of the model problem~\eqref{eq:modelPDE}. Define $U = H^1(\Omega)$, $V_\Omega = H_0^1(\Omega)$ and the operator $B_\Omega \colon U \to V_\Omega'$ to represent the left-hand side of~\eqref{eq:weakform} and let $F_\Omega\in V_\Omega'$ be the right-hand side, i.e., 
\begin{align*}
  (B_\Omega w)(v) &:= b_\Omega(w,v) := \ip{\bA\nabla w}{\nabla v}_\Omega + \ip{\bbeta\cdot\nabla w+c w}{v}_\Omega, \\
  F_\Omega(v) &:= \ip{f}v_\Omega
\end{align*}
for all $w\in U$, $v\in V_\Omega$.
From the well-posedness of the weak formulation~\eqref{eq:weakform} it follows that
\begin{align*}
  \|u-w\|_U \eqsim \|F_\Omega-B_\Omega w\|_{V_\Omega'} + \|F_\Gamma-B_\Gamma w\|_{V_\Gamma'}, 
\end{align*}
with $F_\Gamma$, $B_\Gamma$ and $V_\Gamma$ as discussed in Section~\ref{sec:boucond} above.
We consider two choices of discrete test spaces.
\subsubsection{Lagrange FEM test space}\label{sec:lagrangefem}
We choose $V_{\Omega,\mesh} = P^{k+1}(\mesh)\cap H_0^1(\Omega) \subset V_\Omega$ as the finite dimensional subspace and set, for all $w\in H^1(\Omega)$ such that $\div\bA\nabla w|_T\in L^2(T)$ for all $T\in\mesh$, 
\begin{subequations}\label{eq:defWeakFormEstimators}
\begin{align}
  \eta_{\Omega}(w) &= \sup_{0\neq v\in V_{\Omega,\mesh}} \frac{\ip{f}{v}_\Omega-b_\Omega(w,v)}{\|\nabla v\|_{\Omega}},
  \\
  \rho_{\Omega}(w)^2 &= \sum_{T\in\mesh}h_T^2\|f+\div\bA\nabla w-\bbeta\cdot\nabla w-cw\|_T^2
  + h_T\|\jump{\bA\nabla w\normal}\|_{\partial T\setminus\Gamma}^2, \\
  \eta(w)^2 &= \eta_\Omega(w)^2 + \eta_\Gamma(w)^2, \\
  \rho(w)^2 &= \rho_\Omega(w)^2 + \rho_\Gamma(w)^2.
\end{align}
\end{subequations}
Here, $(\eta_\Gamma$, $\rho_\Gamma)$ denote any of the options described in Section~\ref{sec:boucond}.

\begin{theorem}\label{thm:lagrangeFEM}
  Under the aforegoing definitions~\eqref{eq:defWeakFormEstimators} and assumptions, the estimates
  \begin{align*}
    \eta(w) \lesssim \|u-w\|_{H^1(\Omega)} \lesssim \eta(w) + \rho(w).
  \end{align*}
  hold for all $w\in H^1(\Omega)$, with $\div\bA\nabla w|_T\in L^2(T)$, $T\in\mesh$.
\end{theorem}
\begin{proof}
  We only need to show that
  \begin{align*}
    \eta_\Omega(w) \lesssim \|F_\Omega - B_\Omega w\|_{V_\Omega'} \lesssim \eta_\Omega(w) + \rho_\Omega(w).
  \end{align*}
  First observe that $\Pi_\mesh:= J_{\mesh,0}^{k+1}$ is bounded in $V_\Omega = H_0^1(\Omega)$. Therefore, by the considerations from Section~\ref{sec:abstract}, the lower bound follows. 
  For the upper bound, we set $\widetilde v = v-J_{\mesh,0}^{k+1}v$ and use integration by parts to obtain
  \begin{align*}
    |F_\Omega(\widetilde v)-b_\Omega(w,\widetilde v)| &= \left\lvert\sum_{T\in\mesh}\ip{f+\div\bA\nabla w-\bbeta\cdot\nabla w -cw}{\widetilde v}_T
    -\dual{\bA\nabla w\normal}{\widetilde v}_{\partial T}\right\rvert.
  \end{align*}
  Then, using the (local) approximation properties of the quasi-interpolator and the Cauchy--Schwarz estimate we find that
  \begin{align*}
    &\sum_{T\in\mesh}|\ip{f+\div\bA\nabla w-\bbeta\cdot\nabla w -cw}{\widetilde v}_T| \\&\qquad\lesssim 
    \left(\sum_{T\in\mesh}h_T^2\|f+\div\bA\nabla w-\bbeta\cdot\nabla w -cw\|_T^2\right)^{1/2}\|v\|_{H^1(\Omega)}.
  \end{align*}
  Finally, similar arguments prove that
  \begin{align*}
    \left\vert\sum_{T\in\mesh}\dual{\bA\nabla w\normal}{\widetilde v}_{\partial T}\right\rvert &= \left\lvert\sum_{F\in\faces_\mesh^0} \dual{\jump{\bA\nabla w\normal}}{\widetilde v}_{F}\right\rvert 
    \lesssim \left(\sum_{T\in\mesh} h_T\|\jump{\bA\nabla w\normal}\|_{\partial T\setminus\Gamma}^2\right)^{1/2}\|v\|_{H^1(\Omega)}.
  \end{align*}
  Overall, we have shown that
  \begin{align*}
    \sup_{0\neq v\in V_\Omega} \frac{F_\Omega(v-\Pi_\mesh v)-b_\Omega(w,v-\Pi_\mesh v)}{\|v\|_{V_\Omega}}\lesssim \rho_\Omega(w),
  \end{align*}
  which concludes the proof.
\end{proof}

\begin{remark}\label{rem:weakformsmooth}
  Note that if $w$ is sufficiently smooth, e.g., if $w\in H^2(\Omega)$, then $\div\bA\nabla w\in L^2(\Omega)$ by our assumptions on $\bA$, so that the jump terms in the estimator contribution $\rho_\Omega(w)$ vanish.
  This regularity is ensured, for instance, when $w$ is the output of a neural network with smooth activation functions.
\end{remark}

\begin{remark}\label{rem:BLP}
  If $\div\bA\nabla w\in L^2(\Omega)$, then the volume residual $r := f+\div\bA\nabla w-\bbeta\cdot\nabla w-cw$ belongs to $L^2(\Omega)$ and $\|F_\Omega-B_\Omega w\|_{V_\Omega'} = \|r\|_{H^{-1}(\Omega)}$.
  In this case, Theorem~\ref{thm:lagrangeFEM} states that the $H^{-1}(\Omega)$ norm of $r$ is bounded by the discrete dual norm $\eta_\Omega(w)$ plus the mesh-size weighted $L^2(\Omega)$ norm $\rho_\Omega(w) = \|h_\mesh r\|_\Omega$.
  Estimates of this type have already been used in~\cite{MR1602783} to define discrete $H^{-1}(\Omega)$ inner products for least-squares finite element methods.
\end{remark}

\begin{remark}\label{rem:etaOmegaPrecond}
  In view of efficient computation of $\eta_\Omega$ as discussed in Section~\ref{sec:computelowerbound}, we note that many optimal preconditioners exist for this problem, e.g., the BPX preconditioner or other multilevel techniques, see~\cite{Oswald94}.
\end{remark}

\subsubsection{Bubble enriched test space}\label{sec:weakform:bubbleEnriched}
We choose $V_{\Omega,\mesh} = P^{k+1}(\mesh)\cap H_0^1(\Omega)+P_b^m(\mesh) \subset V_\Omega$ as the finite dimensional subspace and set, for all $w\in H^1(\Omega)$ with $\div\bA\nabla w|_T\in L^2(T)$ for all $T\in\mesh$, 
\begin{subequations}\label{eq:defWeakFormBubbleEstimators}
\begin{align}
  \eta_{\Omega}(w) &= \sup_{0\neq v\in V_{\Omega,\mesh}} \frac{\ip{f}{v}_\Omega-b_\Omega(w,v)}{\|\nabla v\|_{\Omega}}
  \\
  \rho_{\Omega}(w)^2 &= \sum_{T\in\mesh}h_T^2\|(1-\pi_\mesh^m)(f+\div\bA\nabla w-\bbeta\cdot\nabla w-cw)\|_T^2
  + h_T\|\jump{\bA\nabla w\normal}\|_{\partial T\setminus\Gamma}^2, \\
  \eta(w)^2 &= \eta_\Omega(w)^2 + \eta_\Gamma(w)^2, \\
  \rho(w)^2 &= \rho_\Omega(w)^2 + \rho_\Gamma(w)^2.
\end{align}
\end{subequations}
Here, $(\eta_\Gamma$, $\rho_\Gamma)$ denote any of the options described in Section~\ref{sec:boucond}.

\begin{theorem}
  Under the aforegoing definitions~\eqref{eq:defWeakFormBubbleEstimators} and assumptions, the estimates
  \begin{align*}
    \eta(w) \lesssim \|u-w\|_{H^1(\Omega)} \lesssim \eta(w) + \rho(w)
  \end{align*}
  hold for all $w\in H^1(\Omega)$, with $\div\bA\nabla w|_T\in L^2(T)$, $T\in\mesh$.
\end{theorem}
\begin{proof}
  The proof follows the same argumentation as the proof of Theorem~\ref{thm:lagrangeFEM}. 
  The only difference is that we use the operator $\Pi_\mesh = J_{\mesh,0,b}^{k+1,m}$, which allows us to write, for each $T\in\mesh$, the volume residual term as (see Lemma~\ref{lem:szbubble})
  \begin{align*}
    \ip{f+\div\bA\nabla w-\bbeta\cdot\nabla w -cw}{v-\Pi_\mesh v}_T = \ip{(1-\pi_\mesh^m)(f+\div\bA\nabla w-\bbeta\cdot\nabla w -cw)}{v-\Pi_\mesh v}_T.
  \end{align*}
The remaining estimates follow verbatim as in the proof of Theorem~\ref{thm:lagrangeFEM}.
\end{proof}
Remarks~\ref{rem:weakformsmooth} and~\ref{rem:etaOmegaPrecond} apply here as well.
Note that the main difference compared to Section~\ref{sec:lagrangefem} is the presence of the operator $(1-\pi_\mesh^m)$ in the definition of $\rho_\Omega$. This modification reduces the magnitude of the term $\rho_\Omega$, and thus yields a tighter upper bound. 

\subsection{Weak formulation with broken test space}\label{sec:brokenform}
The evaluation of $\eta_\Omega(w)$ for the formulations based upon the weak formulation, as described in Section~\ref{sec:weakform} requires inverting the Riesz matrix associated with (discrete) subspaces of $H_0^1(\Omega)$, or alternatively, evaluating appropriate preconditioners (see Remark~\ref{rem:etaOmegaPrecond}). 
This motivates the study of formulations with broken test spaces, which enable a more efficient realization of the Riesz matrix inversion.
For the broken formulation that we study here, we use 
\begin{align*}
  U = H^1(\Omega)\times H^{-1/2}(\partial\mesh), \quad V_\Omega = H^1(\mesh).
\end{align*}
We define the bilinear form $b_\Omega\colon U\times V_\Omega\to\R$, the induced operator $B_\Omega\colon U\to V_\Omega'$, $(B_\Omega(u,\lambda))(v) := b_\Omega((u,\lambda),v)$, and functional $F_\Omega\in V_\Omega'$ for all $(u,\lambda)\in U$, $v\in V_\Omega$ by
\begin{align*}
  b_\Omega((u,\lambda),v) &:= \ip{\bA\nabla u}{\nabla_\mesh v}_\Omega + \ip{\bbeta\cdot\nabla u+cu}v_\Omega
  - \dual{\lambda}{v}_{\partial\mesh}, \\
  F_\Omega(v) &:= \ip{f}{v}_\Omega.
\end{align*}
For the proof of the following result, we use techniques from~\cite{PrimalDPG13,BreakingSpaces16}.
\begin{lemma}\label{lem:broken}
  The equivalence 
  \begin{align*}
    \|(w,\lambda)\|_U \eqsim \|B_\Omega(w,\lambda)\|_{V_\Omega'} + \|w\|_{1/2,\Gamma}
  \end{align*}
  holds for all $w\in H^1(\Omega)$, and $\lambda\in H^{-1/2}(\partial\mesh)$. In particular, if additionally $\div\bA\nabla w\in L^2(\Omega)$ and $u\in H^1(\Omega)$ is the (unique) solution of~\eqref{eq:weakform}, then
  \begin{align*}
    \|(u,\trdiv\mesh\bA\nabla u)-(w,\trdiv\mesh\bA\nabla w)\|_U \eqsim \|F_\Omega - B_\Omega(w,\trdiv\mesh\bA\nabla w)\|_{V_\Omega'} + \|F_\Gamma-B_\Gamma w\|_{V_\Gamma'},
  \end{align*}
  where $B_\Gamma$, $F_\Gamma$ and $V_\Gamma$ are as in Section~\ref{sec:boucond}.
\end{lemma}
\begin{proof}
  The second asserted equivalence follows directly from the first. To establish the first one, let $(w,\lambda)\in U$ be given. Define $\widetilde w\in H^1(\Omega)$ as the solution of~\eqref{eq:weakform} with $f=0$ and $g=w|_\Gamma$. One verifies that $\widetilde\lambda:= \trdiv\mesh\bA\nabla \widetilde w \in H^{-1/2}(\partial\mesh)$.
  In particular, $(\widetilde w,\widetilde\lambda)\in U$ and, by construction, $B_\Omega(\widetilde w,\widetilde\lambda)=0$. Observe further that $w-\widetilde w\in H_0^1(\Omega)$. Thus, by~\cite{PrimalDPG13}, it follows that
  \begin{align*}
    \|(w-\widetilde w,\lambda-\widetilde\lambda)\|_U \eqsim \|B_\Omega(w-\widetilde w,\lambda-\widetilde\lambda)\|_{V_\Omega'} = \|B_\Omega(w,\lambda)\|_{V_\Omega'}.
  \end{align*}
  Noting that $\|(\widetilde w,\widetilde\lambda)\|_U \lesssim \|\widetilde w\|_{H^1(\Omega)} \lesssim \|w\|_{1/2,\Gamma}$, and applying the triangle inequality, proves that
  \begin{align*}
    \|(w,\lambda)\|_U \leq \|(w-\widetilde w,\lambda-\widetilde\lambda)\|_U + \|(\widetilde w,\widetilde\lambda)\|_U 
    \lesssim \|B_\Omega(w,\lambda)\|_{V_\Omega'} + \|w\|_{1/2,\Gamma}.
  \end{align*}
  The right-hand side is bounded by $\|(w,\lambda)\|_U$ since $B_\Omega$ is bounded and the trace operator $H^1(\Omega)\to H^{1/2}(\Gamma)$ is bounded. This finishes the proof.
\end{proof}
We choose $V_{\Omega,\mesh} = P^k(\mesh) \subset V_\Omega$, for some $k\in\N_0$, as the finite dimensional subspace and set, for all $w\in H^1(\Omega)$ with $\bA\nabla w\in \Hdivset\Omega$, 
\begin{subequations}\label{eq:defBrokenWeakFormEstimators}
  \begin{align}
    \eta_\Omega(w) &= \sup_{0\neq v\in V_{\Omega,\mesh}} \frac{\ip{f}v_\Omega-b_\Omega((w,\trdiv\mesh\bA\nabla w),v)}{\|v\|_{V_\Omega}}, \\
    \rho_\Omega(w)^2 &= \sum_{T\in\mesh} h_T^2\|(1-\pi_\mesh^k)(f+\div\bA\nabla w-\bbeta\cdot\nabla w-cw)\|_T^2, \\
    \eta^2 &= \eta_\Omega^2 + \eta_\Gamma^2, \\
    \rho^2 &= \rho_\Omega^2 + \rho_\Gamma^2.
\end{align}
\end{subequations}
Here, $(\eta_\Gamma$, $\rho_\Gamma)$ denote any of the options described in Section~\ref{sec:boucond}.

\begin{theorem}\label{thm:equivest:brokenweak}
  Under the aforegoing definitions~\eqref{eq:defBrokenWeakFormEstimators} and assumptions, the estimates
  \begin{align*}
\eta(w) \lesssim \|u-w\|_{H^1(\Omega)} + \|\trdiv\mesh\bA\nabla u - \trdiv\mesh\bA\nabla w\|_{-1/2,\partial\mesh} \lesssim \eta(w) + \rho(w)
  \end{align*}
  hold for all $w\in H^1(\Omega)$, with $\bA\nabla w\in \Hdivset\Omega$.
\end{theorem}
\begin{proof}
  Based on Corollary~\ref{cor:equivest} and our previous considerations from Section~\ref{sec:boucond}, it only remains to show that
  \begin{align*}
    \|(F_\Omega-B_\Omega (w,\trdiv\mesh\bA\nabla w))\circ(1-\Pi_{\mesh})\|_{V_{\Omega}'} \lesssim \rho_\Omega(w)
  \end{align*}
  for some bounded operator $\Pi_{\mesh}\colon V_\Omega\to V_{\Omega,\mesh}$. Choosing $\Pi_{\mesh} = \pi_\mesh^k$ (which can be easily seen to be bounded), and employing integration by parts along with the definition of the trace operator, we see that
  \begin{align*}
    &F_\Omega(v-\pi_\mesh^k v) -b_\Omega( (w,\trdiv\mesh\bA\nabla w),v-\pi_\mesh^kv) \\
    &\qquad = \ip{f}{v-\pi_\mesh^k v}_\Omega 
    \\
    &\qquad\qquad-\Big( \ip{\bA\nabla w}{\nabla_\mesh (v-\pi_\mesh^k v)}_\Omega + \ip{\bbeta\cdot\nabla w+cw}{v-\pi_\mesh^k v}_\Omega 
    -\dual{\trdiv\mesh\bA\nabla w}{v-\pi_\mesh^k v}_{\partial\mesh}\Big)
    \\
    &\qquad = \ip{f+\div\bA\nabla w-\bbeta\cdot\nabla w-cw}{v-\pi_\mesh^k v}_\Omega \\
    &\qquad = \ip{(1-\pi_\mesh^k)(f+\div\bA\nabla w-\bbeta\cdot\nabla w-cw)}{v-\pi_\mesh^k v}_\Omega.
  \end{align*}
  Using the local approximation property $\|v-\pi_\mesh^k v\|_T\leq h_T \|\nabla v\|_T$ ($T\in\mesh$) and the Cauchy--Schwarz inequality, we conclude the proof.
\end{proof}
We stress that the computation of all volume terms (i.e., $\eta_\Omega$ and $\rho_\Omega$) is completely local, since the inverse of the Riesz matrix required for $\eta_\Omega$ can be computed element-wise.  
In fact, as the following observation shows, one may even replace $\eta_\Omega$ with simple $L^2$ volume terms.
\begin{theorem}\label{thm:brokenweak:eta}
  Under the assumptions of Theorem~\ref{thm:equivest:brokenweak}, set
  $r_\mesh(w)|_T:= f+\div\bA\nabla w - \bbeta\cdot\nabla w-cw|_T \in L^2(T)$ for all $T\in\mesh$. 
  Then,
  \begin{align*}
    \|\pi_\mesh^0 r_\mesh(w)\|_\Omega \leq \eta_\Omega(w)
    \leq \|\pi_\mesh^0 r_\mesh(w)\|_\Omega + \|h_\mesh(1-\pi_\mesh^0)\pi_\mesh^k r_\mesh(w)\|_\Omega
  \end{align*}
  Furthermore, $\|h_\mesh(1-\pi_\mesh^0)\pi_\mesh^k r_\mesh(w)\|_\Omega \lesssim \eta_\Omega(w)$ so that, in particular, 
  \begin{align*}
    \eta_\Omega(w) \eqsim \|\pi_\mesh^0 r_\mesh(w)\|_\Omega + \|h_\mesh(1-\pi_\mesh^0)\pi_\mesh^k r_\mesh(w)\|_\Omega.
  \end{align*}
\end{theorem}
\begin{proof}
  First, for all $v\in V_{\Omega,\mesh} = P^k(\mesh)$, integration by parts gives
  \begin{align*}
    F_\Omega(v) -b_\Omega( (w,\trdiv\mesh\bA\nabla w),v) = \ip{f+\div\bA\nabla w-\bbeta\cdot\nabla w-cw}{v}_\Omega
= \ip{r_\mesh(w)}v_\Omega.
  \end{align*}
  The lower bound follows by observing that $P^0(\mesh)\subseteq P^k(\mesh)$ and $\|z\|_{V_\Omega} = \|z\|_\Omega$ for $z\in P^0(\mesh)$, yielding
  \begin{align*}
    \eta_\Omega(w) &= \sup_{0\neq v\in P^k(\mesh)} \frac{\ip{r_\mesh(w)}{v}_\Omega}{\|v\|_{V_\Omega}}
    \geq \sup_{0\neq v\in P^0(\mesh)} \frac{\ip{r_\mesh(w)}{v}_\Omega}{\|v\|_\Omega}
    \\
    &= \sup_{0\neq v\in P^0(\mesh)} \frac{\ip{\pi_\mesh^0r_\mesh(w)}{v}_\Omega}{\|v\|_\Omega}
    = \|\pi_\mesh^0r_\mesh(w)\|_\Omega.
  \end{align*}
  For the upper bound decompose $P^k(\mesh)\ni v = \pi_\mesh^0v + (1-\pi_\mesh^0)v$ and note that
  \begin{align*}
    |\ip{r_\mesh(w)}{v}_\Omega| &= |\ip{\pi_\mesh^kr_\mesh(w)}{v}_\Omega| \leq 
    |\ip{\pi_\mesh^kr_\mesh(w)}{\pi_\mesh^0 v}_\Omega| + |\ip{\pi_\mesh^kr_\mesh(w)}{(1-\pi_\mesh^0) v}_\Omega|
    \\
    &= |\ip{\pi_\mesh^0r_\mesh(w)}{\pi_\mesh^0 v}_\Omega| + |\ip{(1-\pi_\mesh^0)\pi_\mesh^kr_\mesh(w)}{(1-\pi_\mesh^0) v}_\Omega|
  \end{align*}
  The first term is bounded by $\|\pi_\mesh^0 r_\mesh(w)\|_\Omega\|v\|_{V_\Omega}$ and the second by using the local approximation property, i.e.,
  \begin{align*}
    |\ip{(1-\pi_\mesh^0)\pi_\mesh^kr_\mesh(w)}{(1-\pi_\mesh^0) v}_\Omega| \leq \|h_\mesh(1-\pi_\mesh^0)\pi_\mesh^kr_\mesh(w)\|_\Omega\|\nabla_\mesh v\|_\Omega.
  \end{align*}
  This implies the second inequality. 
  In addition, note that
  \begin{align*}
    \|h_T\pi_\mesh^k r_\mesh(w)\|_T^2 &= \ip{r_\mesh(w)}{h_T^2\pi_\mesh^kr_\mesh(w)}_T
    \\
    &\lesssim \sup_{0\neq v\in P^k(T)} \frac{\ip{r_\mesh(w)}{v}_T}{\|v\|_{H^1(T)}} h_T^2\|\pi_\mesh^kr_\mesh(w)\|_{H^1(T)}
    \\
    &\lesssim \sup_{0\neq v\in P^k(T)} \frac{\ip{r_\mesh(w)}{v}_T}{\|v\|_{H^1(T)}}\, h_T\|\pi_\mesh^kr_\mesh(w)\|_T
  \end{align*}
  where, in the last estimate, we used the inverse estimate $h_T\|\nabla q\|_T\lesssim \|q\|_T$ for $q\in P^k(T)$ and $h_T\lesssim 1$.
  Dividing by the last term, this shows that
  \begin{align*}
    \|h_T(1-\pi_\mesh^0)\pi_\mesh^k r_\mesh(w)\|_T \leq \|h_T\pi_\mesh^k r_\mesh(w)\|_T \lesssim \sup_{0\neq v\in P^k(T)} \frac{\ip{r_\mesh(w)}{v}_T}{\|v\|_{H^1(T)}}.
  \end{align*}
  Squaring, summing over all $T\in\mesh$ and noting that $\|\cdot\|_{V_{\Omega,\mesh}'}^2 = \sum_{T\in\mesh} \|\cdot\|_{P^k(T)'}^2$ (where, in this notation, $P^k(T)$ is equipped with the $H^1(T)$ norm) proves the last assertion.
\end{proof}

\begin{remark}
  By the definition of $\eta_\Omega$, $\rho_\Omega$, the equivalence from Theorem~\ref{thm:equivest:brokenweak} and properties of the orthogonal projection, it is straightforward to see that
  \begin{align*}
    \eta_\Omega(w)^2 + \rho_\Omega(w)^2 \eqsim \|\pi_\mesh^0 r_\mesh(w)\|_\Omega^2 + \|h_\mesh(1-\pi_\mesh^0) r_\mesh(w)\|_\Omega^2.
  \end{align*}
%  Noteably, the right-hand side is independent of the choice of $k\in\N_0$.
\end{remark}

\subsection{Strong formulation}\label{sec:strongform}
In this section, we discuss a strong formulation of our model problem~\eqref{eq:modelPDE} which can be used to interpret classical PINNs; see the comments in Remark ~\ref{rem:PINN} below. 
To that end, define the space and norm
\begin{align*}
  U&= \set{w\in H^1(\Omega)}{\div\bA\nabla w\in L^2(\Omega)}, \\
  \|w\|_U^2 &= \|w\|_{H^1(\Omega)}^2 + \|\div\bA\nabla w\|_\Omega^2, \quad w\in U
\end{align*}
as well as $V_\Omega := L^2(\Omega)$ with norm $\|\cdot\|_{V_\Omega} := \|\cdot\|_\Omega$.
Define the operator $B_\Omega\colon U \to V_\Omega'$ by
\begin{align*}
  B_\Omega(u)(v) := \ip{-\div\bA\nabla u+\bbeta\cdot\nabla u + cu}v_\Omega
\end{align*}
and $B_\Gamma$, $V_\Gamma$ as discussed in Section~\ref{sec:boucond} above. 
\begin{lemma}\label{lem:strongform}
  The equivalence
  \begin{align*}
    \|w\|_U \eqsim \|B_\Omega w\|_{V_\Omega'} + \|B_\Gamma w\|_{V_\Gamma'}
  \end{align*}
  holds for all $w\in U$. 
\end{lemma}
\begin{proof}
  Let $w\in U$ be given and consider the splitting $w = w_0 + \widetilde w$ where $\widetilde w\in H^1(\Omega)$ solves~\eqref{eq:modelPDE} with $f=0$ and $g=w|_\Gamma$. 
  It follows that $\widetilde w\in U$, thus, $w_0\in U$ and $w_0|_\Gamma = 0$.
  We prove that $\|w_0\|_{U}\lesssim \|B_\Omega w\|_{V_\Omega'}$ and $\|\widetilde w\|_U\lesssim \|B_\Gamma w\|_{V_\Gamma'}$. 
  The latter follows as in the proof of Lemma~\ref{lem:broken}.
  Note that the properties of $\bA$ and the Poincar\'e--Friedrichs inequality imply that $\|w_0\|_{H^1(\Omega)}^2 \eqsim \ip{\bA\nabla w_0}{\nabla w_0}_\Omega$.
  Using integration by parts, $\ip{\bbeta\cdot\nabla w_0+cw_0}{w_0}_\Omega\geq 0$, Cauchy--Schwarz and Young's inequality with parameter $\delta>0$, we infer that
  \begin{align*}
    \|w_0\|_{H^1(\Omega)}^2 &\eqsim \ip{\bA\nabla w_0}{\nabla w_0}_\Omega 
    = \ip{-\div\bA\nabla w_0+\bbeta\cdot\nabla w_0 + cw_0}{w_0}_\Omega -\ip{\bbeta\cdot\nabla w_0+cw_0}{w_0}_\Omega
    \\
    &\leq \ip{-\div\bA\nabla w_0+\bbeta\cdot\nabla w_0 + cw_0}{w_0}_\Omega
    \\
    &\leq \frac{\delta^{-1}}2 \|B_\Omega w_0\|_\Omega^2 + \frac\delta2 \|w_0\|_\Omega^2.
  \end{align*}
  Subtracting the last term for sufficiently small $\delta>0$ and $B_\Omega w_0 = B_\Omega w$ proves that
  \begin{align*}
    \|w_0\|_{H^1(\Omega)} \lesssim \|B_\Omega w\|_{V_\Omega'}.
  \end{align*}
  With the last estimate it also follows that $\|w_0\|_U \lesssim \|B_\Omega w\|_{V_\Omega'}$. 
  This concludes the direction ``$\lesssim$'' of the statement. The other direction follows from the boundedness of the involved operators.
\end{proof}

Set $V_{\Omega,\mesh} = P^k(\mesh) \subset V_\Omega$ as the finite dimensional subspace, and define for $w\in U$
\begin{subequations}\label{eq:defStrongFormEstimators}
  \begin{align}
    \eta_\Omega(w) &= \|\pi_\mesh^k(f+\div\bA\nabla w-\bbeta\cdot\nabla w-cw)\|_\Omega, \\
    \rho_\Omega(w) &= \|(1-\pi_\mesh^k)(f+\div\bA\nabla w-\bbeta\cdot\nabla w-cw)\|_\Omega, \\
    \eta^2 &= \eta_\Omega^2 + \eta_\Gamma^2, \\
    \rho^2 &= \rho_\Omega^2 + \rho_\Gamma^2.
\end{align}
\end{subequations}
Here, $(\eta_\Gamma$, $\rho_\Gamma)$ denote any of the options described in Section~\ref{sec:boucond}.
\begin{theorem}\label{thm:equivest:strongform}
  Under the aforegoing definitions~\eqref{eq:defStrongFormEstimators} and assumptions the estimates
  \begin{align*}
    \eta(w) \lesssim \|u-w\|_{H^1(\Omega)} + \|\div\bA\nabla (u-w)\|_{\Omega} \lesssim \eta(w) + \rho(w).
  \end{align*}
  hold for all $w\in U$.
\end{theorem}
\begin{proof}
  Based on Corollary~\ref{cor:equivest} and our previous considerations from Section~\ref{sec:boucond} as well as Lemma~\ref{lem:strongform} we only need to show that there exists $\Pi_\mesh\colon V\to V_{\Omega,\mesh}$ with $\|\Pi_\mesh\|<\infty$. Clearly, $\Pi_\mesh = \pi_\mesh^k$ does the job in this case and the proof is concluded. 
\end{proof}

\begin{remark}\label{rem:PINN}
  Note that $\eta_{\Omega}^2(w) + \rho_\Omega^2(w) = \|f+\div\bA\nabla w-\bbeta\cdot\nabla w - cw\|_\Omega^2$. 
  Thus, when using $\eta_{\Omega}^2+\rho_{\Omega}^2$ as part of the loss functional in training, this resembles PINNs. 
  However, note that in classical PINNs, a boundary term of the form
  \begin{align*}
    \alpha \|w-g\|_\Gamma^2
  \end{align*}
  is used with a parameter $\alpha>0$ to be chosen by the user. 
  From our considerations from Section~\ref{sec:boucond} we stress that this term only provides a lower bound, $\alpha\|w-g\|_\Gamma^2\lesssim\alpha\|w-g\|_{1/2,\Gamma}^2$, but, in general, not an upper bound of the boundary error. 
  We propose to replace this term by
    \begin{align*}
      \|h_\faces^{-1/2}(w-g)\|_\Gamma^2 + \|h_\faces^{1/2}(1-\pi_\faces^p)\nabla_\Gamma(w-g)\|_\Gamma^2
    \end{align*}
  which is a guaranteed upper bound for $\|w-g\|_{1/2,\Gamma}^2$, see Remark~\ref{rem:investbou}.
\end{remark}
%%%%%%%%%%%%%%%%%%%%%%%%%%%%%%%%%%%%%%%%%%%%%%%%%%%%%%%%%%%%%%%%%%%%%
\section{Numerical examples}\label{sec:numerics}
%%%%%%%%%%%%%%%%%%%%%%%%%%%%%%%%%%%%%%%%%%%%%%%%%%%%%%%%%%%%%%%%%%%%%
In this section, we present the results of our numerical experiments. 
Section~\ref{sec:ex:setup} gives an overview of the general setup and introduces notations used throughout the section. 
In Section~\ref{sec:ex:smooth} we compare training outcomes using different loss functionals derived in Section~\ref{sec:modelpoisson}. 
In Section~\ref{sec:ex:etaVsEtaRho} we study the effect of using either $\eta^2$ or $\eta^2+\rho^2$ as the loss functional. Section~\ref{sec:ex:enforceBC} focuses on the inclusion of boundary conditions in the neural network output and the potential associated slow-down in convergence. 
In Section~\ref{sec:ex:quad}, we discuss the effect of quadrature and propose a simple algorithm to adjust quadrature precision through mesh refinement (Algorithm~\ref{alg}).
In Section~\ref{sec:ex:Lshape}, we consider the approximation of a singular solution by neural networks in a domain with a re-entrant corner. 
In the final experiment (Section~\ref{sec:ex:conv}), we consider a convection-diffusion problem.

We implemented our code in MATLAB 2025a using the Deep Learning Toolbox and executed all programs on a Linux server with an Intel(R) Xeon(R) Platinum 8375C CPU @ 2.90GHz processor.

%%%%%%%%%%%%%%%%%%%%%%%%%%%%%%%%%%%%%%%%%%%%%%
\subsection{General setup}\label{sec:ex:setup}
For all examples, except the one in Section~\ref{sec:ex:conv}, we consider as model PDE the Poisson problem, i.e., $\bA$ is the identity and $\bbeta=0$, $c=0$ in~\eqref{eq:modelPDE}. 

For the numerical examples, we use fully connected feed-forward neural networks $\net$ with $L$ hidden layers and $N$ neurons per layer. 
As activation function, we employ the $\tanh$ function. 
Each network has an input layer with two arguments and a single output.
We use the generic notation $u_\theta = u_\theta(x,y)$ to represent the output of the network, and we stress that $u_\theta$ is smooth. 
Here, $\theta\in\Theta$ denotes all weights and biases that define the neural network. 
We refer to~\cite{lecun2015deep} for an overview. 

For all experiments, we employ the L-BFGS solver (we use the Matlab function \verb+lbfgsupdate+ in each iteration) to find an approximate minimizer of
\begin{align*}
  \theta^* = \argmin_{\theta\in\Theta} \loss_\star(u_\theta)
\end{align*}
where $\loss_\star$ denotes one of the loss functions defined below. 

For the first three losses we define $\eta_\Gamma$, $\rho_\Gamma$ as in~\eqref{eq:defetarhobou1} with $p=0$ (Section~\ref{sec:boucond:opt1}).
In particular, we define the boundary mesh $\faces$ to be the restriction of $\mesh$ onto $\Gamma$.

{\emph{Weak bubble loss.}}
\begin{align}\label{eq:defloss:wb}
  \loss_\mathrm{wb} = \eta_\Omega^2 + \eta_\Gamma^2 + \rho_\Omega^2 + \rho_\Gamma^2
  \quad\text{with $\eta_\Omega$, $\rho_\Omega$ defined in~\eqref{eq:defWeakFormBubbleEstimators} with $k=0=m$}.
\end{align}

{\emph{Broken loss.}}
\begin{align}\label{eq:defloss:broken}
  \loss_\mathrm{br} = \eta_\Omega^2 + \eta_\Gamma^2 + \rho_\Omega^2 + \rho_\Gamma^2
  \quad\text{with $\eta_\Omega$, $\rho_\Omega$ defined in~\eqref{eq:defBrokenWeakFormEstimators} with $k=1$}.
\end{align}

{\emph{Modified PINN loss.}}
\begin{align}\label{eq:defloss:PINNmod}
  \loss_\mathrm{Pmod} = \eta_\Omega^2 + \eta_\Gamma^2 + \rho_\Omega^2 + \rho_\Gamma^2
  \quad\text{with $\eta_\Omega$, $\rho_\Omega$ defined in~\eqref{eq:defStrongFormEstimators} with $k=0$}.
\end{align}

{\emph{Classical PINN loss.}}
\begin{align}\label{eq:defloss:PINN}
  \loss_\mathrm{PINN} = \frac{|\Omega|}{N_\Omega}\sum_{j=1}^{N_\Omega} |(f+\Delta u_\theta)(x_j,y_j)|^2 + \frac{\alpha}{N_\Gamma} \sum_{k=1}^{N_\Gamma} |u_\theta(x^b_k,y^b_k)|^2
\end{align}
where $\{(x_j,y_j)\}_{j=1}^{N_\Omega}$ and $\{(x^b_k,y^b_k)\}_{k=1}^{N_\Gamma}$ denote randomly sampled (uniform distribution) points in $\Omega$ and $\Gamma$, respectively. 
The parameter $\alpha>0$ has to be chosen by the user. We choose $\alpha = N_\Gamma$. This mimics a negative power of the mesh size if the points $(x_k^b,y_k^b)$ were the nodes of a uniform partition of $\Gamma$, see Remark~\ref{rem:PINN}.

\subsection{Comparison between methods for smooth solution}\label{sec:ex:smooth}
Let $\Omega = (0,1)^2$. We consider the exact solution $u(x,y) = \sin(\pi x)\sin(\pi y)$, $(x,y)\in\Omega$, i.e., $f = -\Delta u = 2\pi^2\sin(\pi x)\sin(\pi y)$ and $g=0$.
We consider a partition of $\Omega$ into $64$ triangles with equal area and $h_\mesh|_T = \tfrac14$ for all $T\in\mesh$.
The same initial net with $L=5$ layers and $N=20$ neurons per layer, which results in $1761$ total learnable parameters, is trained using the losses $\loss_\mathrm{wb}$, $\loss_\mathrm{br}$, $\loss_\mathrm{Pmod}$ and $\loss_\mathrm{PINN}$ as defined above, 
where for the first three we use a $6$-point quadrature rule on triangles (exact of degree $4$) and a $4$-point Gaussian quadrature rule on boundary elements.
The number of sample points in $\Omega$ (resp. $\Gamma$) for evaluation of $\loss_\mathrm{PINN}$ is the same as the number of quadrature points in $\Omega$ (resp. $\Gamma$) used for the other loss functionals.

In Figure~\ref{fig:smooth} we visualize the $H^1(\Omega)$ errors over accumulated solving time (left) and iterations (right). 
With respect to both time and iteration, our three proposed losses perform better than the classical PINN strategy. 
Concerning the losses~\eqref{eq:defloss:wb}--\eqref{eq:defloss:PINNmod} we see that $\loss_\mathrm{wb}$ performs best, while the other two are comparable. Note that $\loss_\mathrm{br}$ and $\loss_\mathrm{Pmod}$ measure the error in stronger norms than the $H^1(\Omega)$ norm.
The bottom row (left plot) shows the ratio $\sqrt{\loss_\star}/\|u-u_\theta\|_{H^1(\Omega)}$ when using the different losses. 
For the first three losses we observe that the ratio is around $2$ whereas in the case of using $\loss_\mathrm{PINN}$ a higher variation is seen, indicating less robustness for the classical PINN loss.
The bottom row also shows the different contributions that make up $\loss_\mathrm{wb}$. All contributions seem to decrease at the same speed.

\begin{figure}
  \begin{tikzpicture}
  \begin{groupplot}[group style={group size= 2 by 1},width=0.5\textwidth,cycle list/Dark2-6,
      cycle multiindex* list={
        mark list*\nextlist
      Dark2-6\nextlist},
      every axis plot/.append style={ultra thick},
      grid=major,
    ]
    \nextgroupplot[ymode=log,%xmode=log,
      legend entries={\tiny{weak bubble},\tiny{broken},\tiny{PINN mod.},\tiny{PINN}},
    legend pos=north east,xlabel={time (s)},ylabel={\small{$\|u-u_\theta\|_{H^1(\Omega)}$}}]
    \addplot table [x=timeSolve,y=errH1] {data/SmoothWB.dat};
    \addplot table [x=timeSolve,y=errH1] {data/SmoothBroken.dat};
    \addplot table [x=timeSolve,y=errH1] {data/SmoothPINNmod.dat};
    \addplot table [x=timeSolve,y=errH1] {data/SmoothPINN.dat};
    \nextgroupplot[ymode=log,%xmode=log,
      legend entries={\tiny{weak bubble},\tiny{broken},\tiny{PINN mod.},\tiny{PINN}},
    legend pos=north east,xlabel={iterations}]
    \addplot table [x=iterations,y=errH1] {data/SmoothWB.dat};
    \addplot table [x=iterations,y=errH1] {data/SmoothBroken.dat};
    \addplot table [x=iterations,y=errH1] {data/SmoothPINNmod.dat};
    \addplot table [x=iterations,y=errH1] {data/SmoothPINN.dat};
  \end{groupplot}
\end{tikzpicture}
\begin{tikzpicture}
  \begin{groupplot}[group style={group size= 2 by 1},width=0.5\textwidth,cycle list/Dark2-6,
      cycle multiindex* list={
        mark list*\nextlist
      Dark2-6\nextlist},
      every axis plot/.append style={ultra thick},
      grid=major,
    ]
    \nextgroupplot[%ymode=log,%xmode=log,
      legend entries={\tiny{weak bubble},\tiny{broken},\tiny{PINN mod.},\tiny{PINN}},
    legend pos=north east,xlabel={iterations},title={\small{$\sqrt{\loss_\star}/\|u-u_\theta\|_{H^1(\Omega)}$}}]
    \addplot table [x=iterations,y=ratLossErr] {data/SmoothWB.dat};
    \addplot table [x=iterations,y=ratLossErr] {data/SmoothBroken.dat};
    \addplot table [x=iterations,y=ratLossErr] {data/SmoothPINNmod.dat};
    \addplot table [x=iterations,y=ratLossErr] {data/SmoothPINN.dat};
    \nextgroupplot[ymode=log,%xmode=log,
      legend entries={\tiny{$\eta_\Omega$, wb},\tiny{$\eta_\Gamma$, wb},\tiny{$\rho_\Omega$, wb},\tiny{$\rho_\Gamma$, wb}},
    legend pos=north east,xlabel={iterations},title={\small{Contributions of $\loss_\mathrm{wb}$}}]
    \addplot table [x=iterations,y=etaVol] {data/SmoothWB.dat};
    \addplot table [x=iterations,y=etaBou] {data/SmoothWB.dat};
    \addplot table [x=iterations,y=rhoVol] {data/SmoothWB.dat};
    \addplot table [x=iterations,y=rhoBou] {data/SmoothWB.dat};
  \end{groupplot}
\end{tikzpicture}
  \caption{First row shows error in $H^1(\Omega)$ norm over time (left) and number of iterations (right). The second row shows the ratio between the square root of the loss and the error in the $H^1(\Omega)$ norm (left) and the individual contributions of $\loss_\mathrm{wb}$ (right).}\label{fig:smooth}
\end{figure}
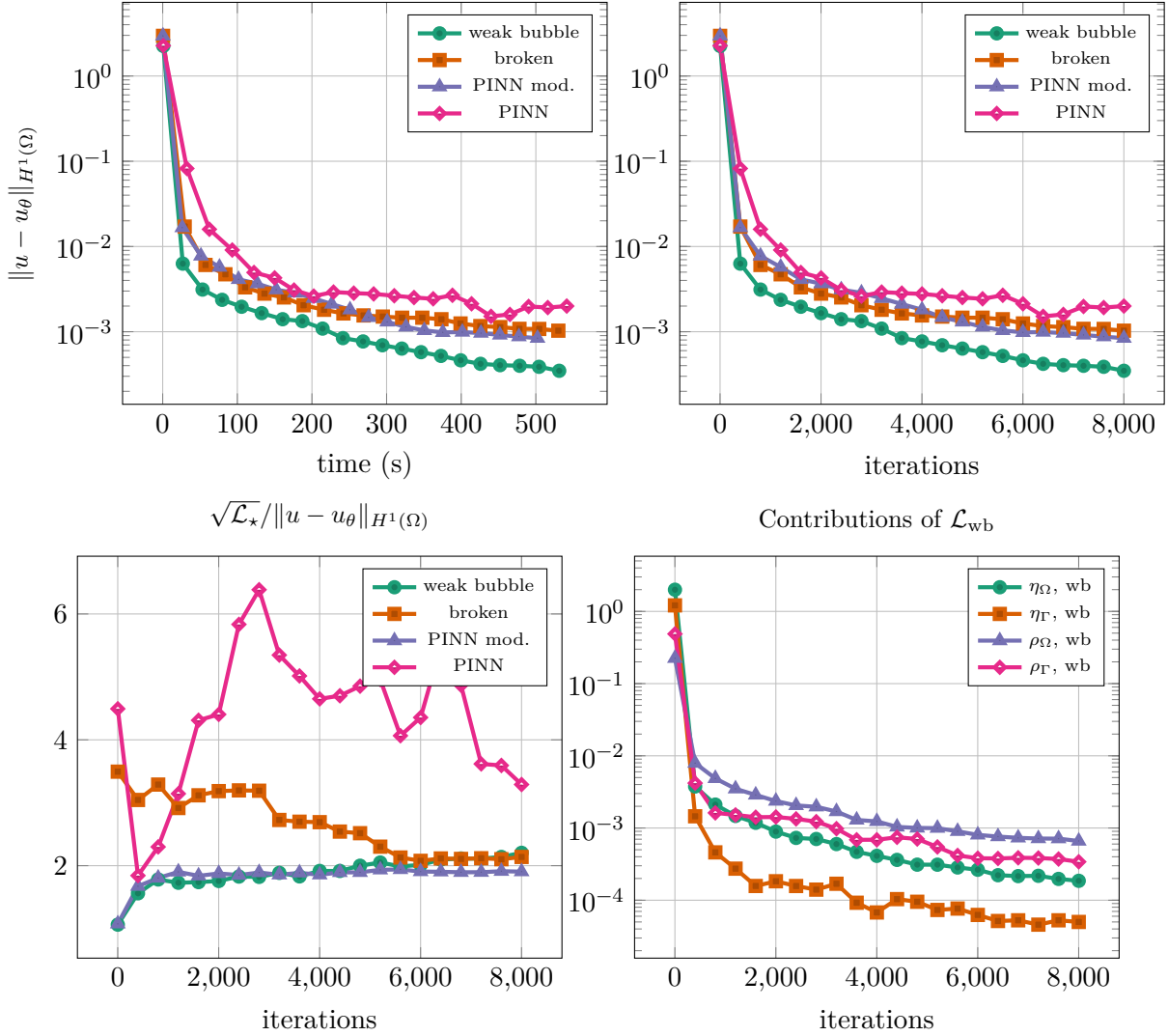

\subsection{Using $\eta^2$ or $\eta^2+\rho^2$ as loss}\label{sec:ex:etaVsEtaRho}
For this example, we consider the same setup as in Section~\ref{sec:ex:smooth}.
We want to compare the training performance using only the residual contribution $\eta^2$ as the loss versus considering $\eta^2+\rho^2$.
For that we consider $\loss_\mathrm{wb}$ as before and define $\loss_{\mathrm{wb},\eta}=\eta_\Omega^2+\eta_\Gamma^2$ with $\eta_\Omega$, $\eta_\Gamma$ as in~\eqref{eq:defloss:wb}.
It has been observed in~\cite{RVPINN24} that using $\loss_{\mathrm{wb},\eta}$ as  loss functional can lead to a gap between $\loss_{\mathrm{wb},\eta}$ and the exact error $\|u-u_\theta\|_{H^1(\Omega)}$. We observe a similar effect in our experiment, see Figure~\ref{fig:EtaVsEtaRho}. In particular, in the right plot, it can be seen that the ratio $\sqrt{\loss_{\mathrm{wb},\eta}}/\|u-u_\theta\|_{H^1(\Omega)}$ seemingly becomes smaller.
Furthermore, the left plot shows that using $\loss_\mathrm{wb}$ clearly outperforms using only $\loss_{\mathrm{wb},\eta}$ in terms of error over iterations.

\begin{figure}
  \begin{tikzpicture}
  \begin{groupplot}[group style={group size= 2 by 1},width=0.5\textwidth,cycle list/Dark2-6,
      cycle multiindex* list={
        mark list*\nextlist
      Dark2-6\nextlist},
      every axis plot/.append style={ultra thick},
      grid=major,
    ]
    \nextgroupplot[ymode=log,%xmode=log,
      legend entries={\tiny{$\|u-u_\theta\|_{H^1(\Omega)}$ using $\loss_\mathrm{wb}$},\tiny{$\|u-u_\theta\|_{H^1(\Omega)}$ using $\loss_{\mathrm{wb},\eta}$},\tiny{$\sqrt{\loss_\mathrm{wb}}$},\tiny{$\sqrt{\loss_{\mathrm{wb},\eta}}$}},
      legend pos=north east,xlabel={iterations}]
    \addplot table [x=iterations,y=errH1] {data/EtaVsEtaRho_WB_EtaRho.dat};
    \addplot table [x=iterations,y=errH1] {data/EtaVsEtaRho_WB_Eta.dat};
    \addplot table [x=iterations,y=loss] {data/EtaVsEtaRho_WB_EtaRho.dat};
    \addplot table [x=iterations,y=loss] {data/EtaVsEtaRho_WB_Eta.dat};
    \nextgroupplot[%ymode=log,%xmode=log,
      legend entries={%\tiny{using $\loss_\mathrm{wb}$},
      \tiny{$\sqrt{\loss_{\mathrm{wb},\eta}}/\|u-u_\theta\|_{H^1(\Omega)}$}},
    legend pos=north east,xlabel={iterations}%,title={\small{$\sqrt{\loss_\star}/\|u-u_\theta\|_{H^1(\Omega)}$}}
    ]
    %\addplot table [x=iterations,y=ratLossErr] {data/EtaVsEtaRho_WB_EtaRho.dat};
    \addplot table [x=iterations,y=ratLossErr] {data/EtaVsEtaRho_WB_Eta.dat};
  \end{groupplot}
\end{tikzpicture}
  \caption{The left plot shows the $H^1(\Omega)$ error and the square root of the loss over number of iterations. The right plot shows the ratio between the square root of the loss and the error in the $H^1(\Omega)$ norm.}\label{fig:EtaVsEtaRho}
\end{figure}
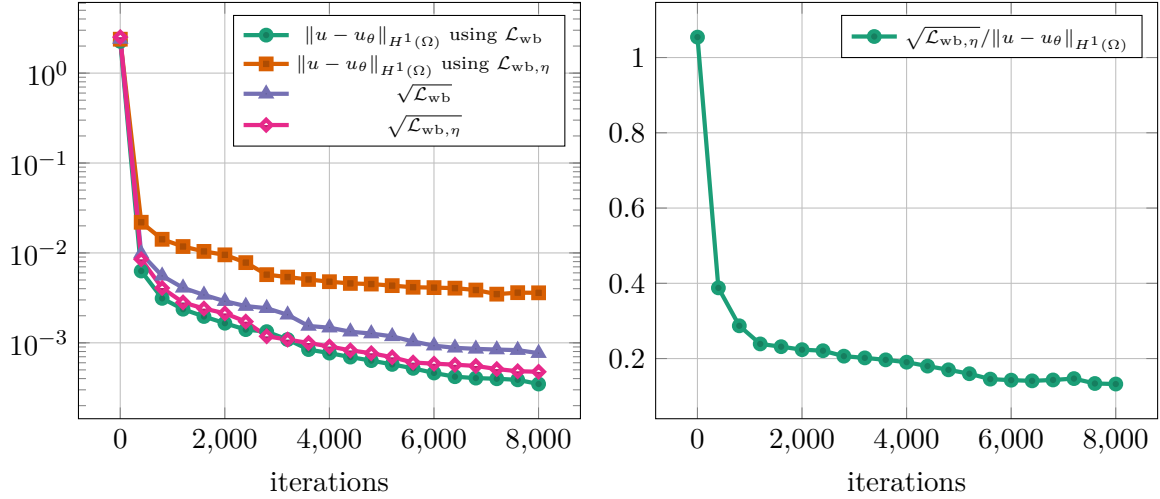

\subsection{Comparison to enforcing boundary conditions}\label{sec:ex:enforceBC}
For this experiment, consider enforcing boundary conditions by multiplying the network output with a fixed function $\chi$. 
In this example we choose $\chi(x,y) = x(1-x)y(1-y)$ which is smooth, $\chi(x,y)>0$ in $\Omega=(0,1)^2$ and vanishes on the boundary. 
We shall refer to this output as $\widetilde u_{\theta}$, whereas $u_\theta$ refers to the output of a neural network without multiplying by $\chi$.
We use the exact manufactured solution
\begin{align*}
  u(x,y) = \big(1-e^{-x/\varepsilon}\big)(1-x)\sin(\pi y) \quad (x,y)\in \Omega
\end{align*}
for $\varepsilon=10^{-2}$ and compute $f=-\Delta u$.
We stress that this function exhibits a boundary layer near the edge $x=0$ of width $\varepsilon$.

We consider a neural network with $L=5$ layers and $N=30$ neurons per layer, which gives a total of $3841$ learnable parameters. 
As background mesh we use the one depicted in Figure~\ref{fig:enforceBC}.
It is pre-refined in order to resolve the data $f$ (which also has a layer). 
In total it has $3648$ volume elements and, using the same quadrature rule as for the previous experiments, we obtain $21888$ quadrature nodes in the volume and $480$ on the boundary.
We then compare the errors generated by using either $\loss_\mathrm{wb}(u_\theta)$ or 
\begin{align*}
  \loss_\mathrm{wb}(\widetilde u_\theta) = \eta_\Omega^2(\widetilde u_\theta) + \rho_\Omega^2(\widetilde u_\theta)
\end{align*}
as loss functionals.
Note that $\eta_\Gamma(\widetilde u_\theta) = 0 = \rho_\Gamma(\widetilde u_\theta)$ since $\widetilde u_\theta$ vanishes on $\Gamma$.
The results of our numerical simulations are presented in Figure~\ref{fig:enforceBC}. 
One observes that enforcing boundary conditions by multiplication with a smooth function can drastically reduce the speed of learning. 
In particular, at iteration $3000$ the $H^1(\Omega)$ error when using $\loss_\mathrm{wb}(u_\theta)$ is almost a factor $100$ smaller than when using $\loss_\mathrm{wb}(\widetilde u_\theta)$.
\begin{figure}
  \begin{tikzpicture}
  \begin{groupplot}[group style = {group size = 2 by 1},width=0.5\textwidth,%,height=0.15\textheight
    ]
    \nextgroupplot
    \addplot[patch,color=white,hide axis,%mesh,
          faceted color = black, line width = 0.3pt,
          patch table ={data/EnforceBCmesh_ele.dat}] file{data/EnforceBCmesh_coo.dat};
    \nextgroupplot[ymode=log,%xmode=log,
      legend entries={\tiny{enforcing bc},\tiny{not enforcing bc}},
      legend pos=north east,xlabel={iterations}]
      \addplot table [x=iterations,y=errH1] {data/EnforceBC.dat};
      \addplot table [x=iterations,y=errH1] {data/NotEnforceBC.dat};
  \end{groupplot}
\end{tikzpicture}
  \caption{Comparing the inclusion of ``enforcing boundary conditions'' by multiplying network output with $\chi$ (Section~\ref{sec:ex:enforceBC}).}\label{fig:enforceBC}
\end{figure}
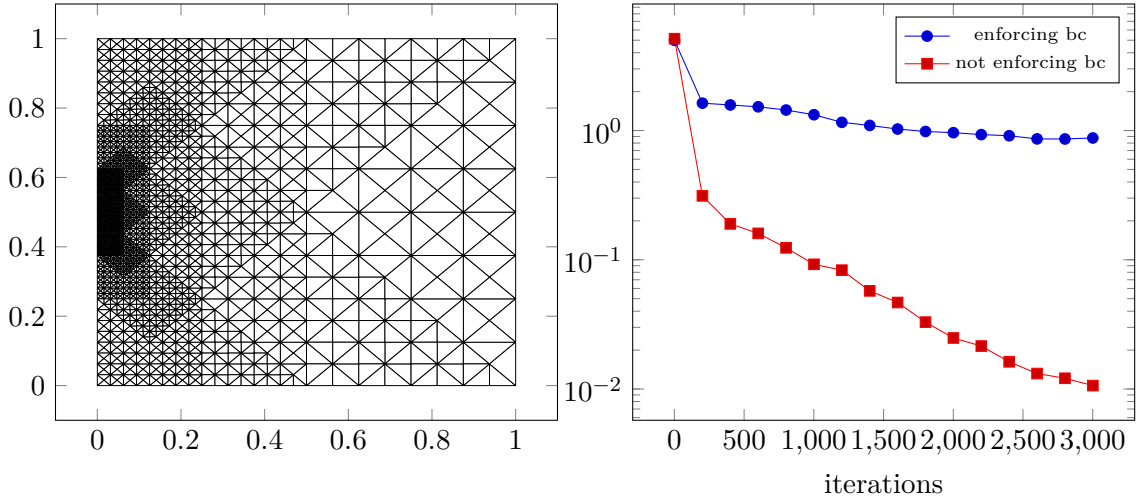

\subsection{Effect of quadrature and adaptive algorithm}\label{sec:ex:quad}

%\IncMargin{1.5em}
\begin{algorithm}[!ht]
\SetKwInOut{Input}{Input}
\SetKwInOut{Output}{Output}
\Input{Neural network $\net$, mesh $\mesh$, loss $\loss$, stopping criterion, adaptivity parameters $\tau_1\in(0,1)$, $\tau_2\in(0,1)$}
\Output{Trained neural network $\net$, refined mesh $\mesh$}
\BlankLine
  \While{stopping criterion not met}
  {
    Update neural network $\net$ by doing one solver step\;
    Compute $\loss(u_\theta)$ and $\widehat\loss(u_\theta)$\;
    \If{$|\loss(u_\theta)-\widehat\loss(u_\theta)|>\tau_1\widehat\loss(u_\theta)$}
    {
      $\delta_{\max} = \max\{|\loss(u_\theta)[T]-\widehat\loss(u_\theta)[T]|/\widehat\loss(u_\theta)[T]\,:\,T\in\mesh\}$\;
      $\mathcal{M} = \set{T\in\mesh}{|\loss(u_\theta)[T]-\widehat\loss(u_\theta)[T]|>\tau_2 \widehat\loss(u_\theta)[T] \delta_{\max}}$\;
      Update mesh $\mesh$ by refining (at least) all marked elements $\mathcal{M}$\;
    }
  }
\caption{Training algorithm with mesh-refinement}\label{alg}
\end{algorithm}
%\DecMargin{1.5em}

For this experiment we study quadrature and how to locally refine the mesh in order to increase the number of quadrature points in a specific region.
On an element $T\in\mesh$ we use, as in the previous experiments, a $6$-point quadrature rule and, on $F\in\faces$, we use a $4$-point Gaussian quadrature. 
The totality of all quadrature nodes on elements and boundary elements is denoted by $Q_{\mesh,\faces}$ 
and $\widehat Q_{\mesh,\faces}$ denotes a higher order quadrature rule. Specifically, we choose a $16$-point quadrature for volume elements (exact for degree $8$) and an $8$-point Gaussian quadrature rule for boundary elements.
We consider the loss $\loss_\mathrm{wb}$ and stress that it can be written in a localized way as
\begin{align*}
  \loss_\mathrm{wb}(u_\theta) &= \eta_\Omega^2 + \rho_\Omega^2 + \eta_\Gamma^2 + \rho_\Gamma^2
  = \sum_{T\in\mesh} \loss_\mathrm{wb}(u_\theta)[T].
\end{align*}
For estimators $\rho_\Omega^2$ and $\rho_\Gamma^2$ this decomposition is clear since they are defined using $L^2$ norms. 
(We associate to each boundary element $F\in\faces$ its unique element $T_F\in\mesh$.)
For the estimators $\eta_\Omega^2$, $\eta_\Gamma^2$ we note that there exist unique discrete functions $r_\Omega\in V_{\Omega,\mesh}\subset H_0^1(\Omega)$, $r_\Gamma\in V_{\Gamma,\widetilde\mesh} \subset\Hdivset\Omega$ with
\begin{align*}
  \eta_\Omega^2 = \|\nabla r_\Omega\|_{\Omega}^2, \quad \eta_\Gamma^2 = \|r_\Gamma\|_{\Hdivset\Omega}^2
\end{align*}
which can then be decomposed into local contributions.

To compute $\loss_\mathrm{wb}$ we use quadrature $Q_{\mesh,\faces}$ and define $\widehat\loss_\mathrm{wb}$ analogously but using the higher order quadrature $\widehat Q_{\mesh,\faces}$.

For the experiment we consider the same setup as in Section~\ref{sec:ex:smooth} but we choose as initial mesh a triangulation $\mesh$ of $\Omega$ into $4$ equal-area triangles.
Using $\loss_\mathrm{wb}$ we have $24$ quadrature points in $\Omega$ and $16$ on the boundary. 
In Figure~\ref{fig:quad} we visualize the error as well as the loss (left plot) and the ratio thereof (right plot), all labeled as \emph{no adap.}
One observes that training could be stopped after a few iterations, after which the error does not improve anymore.
Interestingly, the loss does not decrease either so that the ratio becomes constant. However, the loss is smaller by a significant factor, indicating that quadrature is possibly not sufficiently accurate, and, therefore, robustness is lost.
To overcome this, one could use ``overkill'' quadrature rules.
However, we want to avoid this and propose a simple adaptive algorithm sketched in Algorithm~\ref{alg}.
The basic idea is the following: In each iteration, we compute $\loss_\mathrm{wb}$ and $\widehat\loss_\mathrm{wb}$. If they are relatively far apart (measured by the parameter $\tau_1$), then elements are marked for mesh-refinement.
Here, we use a maximum strategy where we mark all elements where the relative local error between $\loss_\mathrm{wb}$ and $\widehat\loss_\mathrm{wb}$ is larger than a certain percentage (parameter $\tau_2$) of the maximum of the local errors.
We use the newest vertex bisection algorithm to iteratively bisect each marked element into $4$ refined elements. 
This then quadruples the number of quadrature points in the same region.

Employing the adaptive algorithm with parameters $\tau_1= 0.3$, $\tau_2 = 0.7$ we find that the ratio between loss and error is around $1.4$, see results presented in Figure~\ref{fig:quad} labeled \emph{adap.}
Furthermore, one observes that the error does not become stationary during the same number of iterations and keeps getting smaller. 
We note that at iteration $4000$ the mesh generated by Algorithm~\ref{alg} has $387$ elements, resulting in $2322$ quadrature nodes in $\Omega$ and $148$ quadrature nodes on $\Gamma$.

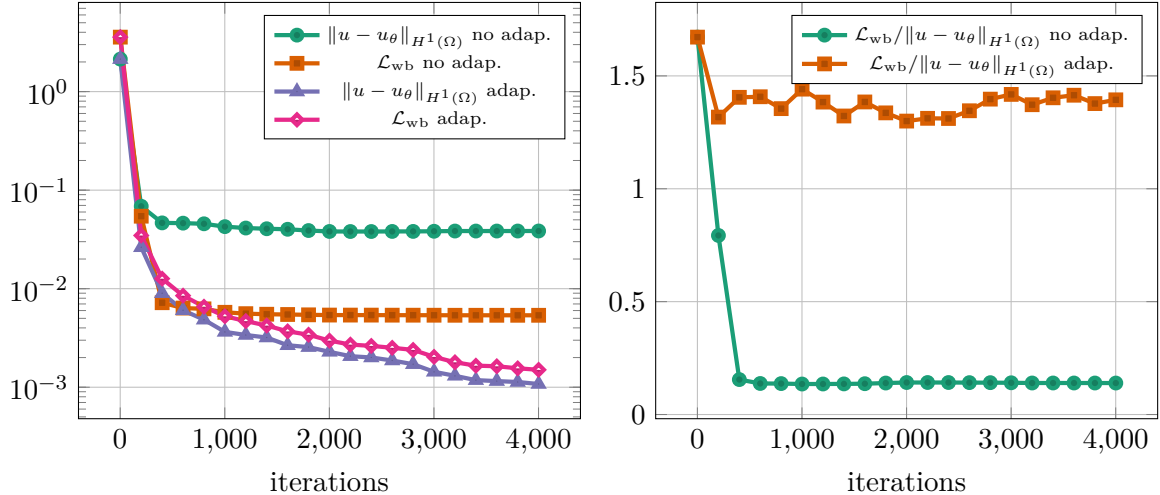
\begin{figure}
  \begin{tikzpicture}
  \begin{groupplot}[group style={group size= 2 by 1},width=0.5\textwidth,cycle list/Dark2-6,
      cycle multiindex* list={
        mark list*\nextlist
      Dark2-6\nextlist},
      every axis plot/.append style={ultra thick},
      grid=major,
    ]
    \nextgroupplot[ymode=log,%xmode=log,
      legend entries={\tiny{$\|u-u_\theta\|_{H^1(\Omega)}$ no adap.},\tiny{$\sqrt{\loss_\mathrm{wb}}$ no adap.},\tiny{$\|u-u_\theta\|_{H^1(\Omega)}$ adap.},\tiny{$\sqrt{\loss_\mathrm{wb}}$ adap.}},
    legend pos=north east,xlabel={iterations}]
    \addplot table [x=iterations,y=errH1] {data/NoAdaptivitySmooth.dat};
    \addplot table [x=iterations,y=errEst] {data/NoAdaptivitySmooth.dat};
    \addplot table [x=iterations,y=errH1] {data/AdaptivitySmooth.dat};
    \addplot table [x=iterations,y=errEst] {data/AdaptivitySmooth.dat};
    \nextgroupplot[%ymode=log,%xmode=log,
      legend entries={\tiny{$\sqrt{\loss_\mathrm{wb}}$$/\|u-u_\theta\|_{H^1(\Omega)}$ no adap.},\tiny{$\sqrt{\loss_\mathrm{wb}}$$/\|u-u_\theta\|_{H^1(\Omega)}$ adap.}},
    legend pos=north east,xlabel={iterations}]
    \addplot table [x=iterations,y=ratLossErr] {data/NoAdaptivitySmooth.dat};
    \addplot table [x=iterations,y=ratLossErr] {data/AdaptivitySmooth.dat};
  \end{groupplot}
\end{tikzpicture}
  \caption{Results for setup of Section~\ref{sec:ex:quad} with (adap.) and without (no adap.) using Algorithm~\ref{alg}.}\label{fig:quad}
\end{figure}

Another effect of refining the mesh is that $\rho_\Omega$, $\rho_\Gamma$ become smaller (they are multiplied by the mesh size) and $\eta_\Omega$, $\eta_\Gamma$ become larger since the discrete spaces defined over the mesh grow in dimension. 
Therefore, the bound $\eta(u_\theta)\lesssim \|u-u_\theta\|_{H^1(\Omega)}$ becomes tighter.
For the upper bound $\eta(u_\theta)+\rho(u_\theta)$, eventually, $\rho(u_\theta)$ should become less dominant.

\subsection{Singular solution}\label{sec:ex:Lshape}
For the last experiment consider the domain $\widetilde\Omega = (-1,1)^2\setminus[-1,0]^2$ and define $\Omega$ by rotating $\widetilde\Omega$ $45$ degrees clockwise. Define the exact solution $u = r^{2/3}\cos(\tfrac23\varphi)$, where $(r,\varphi)$ denote polar coordinates centered at the origin and such that $u$ vanishes on the two incoming edges to the re-entrant corner.
We compute $f=-\Delta u = 0$, and $g= u|_\Gamma$.

We consider a neural network with $L=8$ layers and $N=20$ neurons per layer which has in total $3021$ learnable parameters.
As loss functional we take $\loss_\mathrm{wb}$ and employ Algorithm~\ref{alg} with a uniform initial mesh of $192$ elements and $\tau_1=0.2$, $\tau_2=0.75$. The results are shown in Figure~\ref{fig:Lshape} labeled \emph{adap.}
The final mesh at $12000$ iterations consists of $567$ elements with $3402$ quadrature nodes in $\Omega$ and $172$ quadrature nodes on $\Gamma$.
It is visualized in Figure~\ref{fig:Lshape} (bottom row on the right). The bottom row (left plot) also visualizes the contributions that make up $\loss_\mathrm{wb}$. One sees that all contributions except $\eta_\Gamma$, which is smaller by a factor of about $3$--$4$, are of the same order.
We compare these results (also in Figure~\ref{fig:Lshape}) to using a fixed mesh with $768$ equal-area triangles, resulting in $4608$ quadrature nodes in $\Omega$ and $256$ quadrature nodes on $\Gamma$ (labeled \emph{no adap.}).
One observes the loss of robustness without using Algorithm~\ref{alg} although the total number of quadrature nodes is larger. 
In particular, note that we use a logarithmic scaling of the vertical axis in the right plot.
The ratio between loss and exact error hovers around $1.2$ when using the adaptive algorithm, but without employing the algorithm, it deteriorates to about $0.01$.
This indicates that it is essential to locally increase the number of quadrature points for problems with local features such as singular gradients.

\begin{figure}
  \begin{tikzpicture}
  \begin{groupplot}[group style={group size= 2 by 1},width=0.5\textwidth,cycle list/Dark2-6,
      cycle multiindex* list={
        mark list*\nextlist
      Dark2-6\nextlist},
      every axis plot/.append style={ultra thick},
      grid=major,
    legend style={at={(0.5,1.05)},anchor=south},
    legend columns=1, 
    legend style={
                    % the /tikz/ prefix is necessary here...
                    % otherwise, it might end-up with `/pgfplots/column 2`
                    % which is not what we want. compare pgfmanual.pdf
        /tikz/column 2/.style={
          column sep=5pt,
      }},
    ]
    \nextgroupplot[ymode=log,%xmode=log,
      legend entries={\tiny{$\|u-u_\theta\|_{H^1(\Omega)}$ no adap.},\tiny{$\sqrt{\loss_\mathrm{wb}}$ no adap.},\tiny{$\|u-u_\theta\|_{H^1(\Omega)}$ adap.},\tiny{$\sqrt{\loss_\mathrm{wb}}$ adap.}},
    %legend pos=north east,
    xlabel={iterations}]
    \addplot table [x=iterations,y=errH1] {data/NoAdaptivityLshape.dat};
    \addplot table [x=iterations,y=errEst] {data/NoAdaptivityLshape.dat};
    \addplot table [x=iterations,y=errH1] {data/AdaptivityLshape.dat};
    \addplot table [x=iterations,y=errEst] {data/AdaptivityLshape.dat};
    \nextgroupplot[ymode=log,%xmode=log,
      legend entries={\tiny{$\sqrt{\loss_\mathrm{wb}}/\|u-u_\theta\|_{H^1(\Omega)}$ no adap.},\tiny{$\sqrt{\loss_\mathrm{wb}}/\|u-u_\theta\|_{H^1(\Omega)}$ adap.}},
    %legend pos=north east,
    xlabel={iterations}]
    \addplot table [x=iterations,y=ratLossErr] {data/NoAdaptivityLshape.dat};
    \addplot table [x=iterations,y=ratLossErr] {data/AdaptivityLshape.dat};
  \end{groupplot}
\end{tikzpicture}
\begin{tikzpicture}
  \begin{groupplot}[group style = {group size = 2 by 1},width=0.5\textwidth,%,height=0.15\textheight
    ]
    \nextgroupplot[ymode=log,%xmode=log,
      legend entries={\tiny{$\eta_\Omega$, adap.},\tiny{$\eta_\Gamma$, adap.},\tiny{$\rho_\Omega$, adap.},\tiny{$\rho_\Gamma$, adap.}},
    legend pos=north east,xlabel={iterations}]
    \addplot table [x=iterations,y=etaVol] {data/AdaptivityLshape.dat};
    \addplot table [x=iterations,y=etaBou] {data/AdaptivityLshape.dat};
    \addplot table [x=iterations,y=rhoVol] {data/AdaptivityLshape.dat};
    \addplot table [x=iterations,y=rhoBou] {data/AdaptivityLshape.dat};
    \nextgroupplot
    \addplot[patch,color=white,axis equal,hide axis,%mesh,
          faceted color = black, line width = 0.3pt,
          patch table ={data/LshapeMesh_ele.dat}] file{data/LshapeMesh_coo.dat};
  \end{groupplot}
\end{tikzpicture}
  \caption{Results for problem on L-shaped domain with singular solution (Section~\ref{sec:ex:Lshape}).}\label{fig:Lshape}
\end{figure}
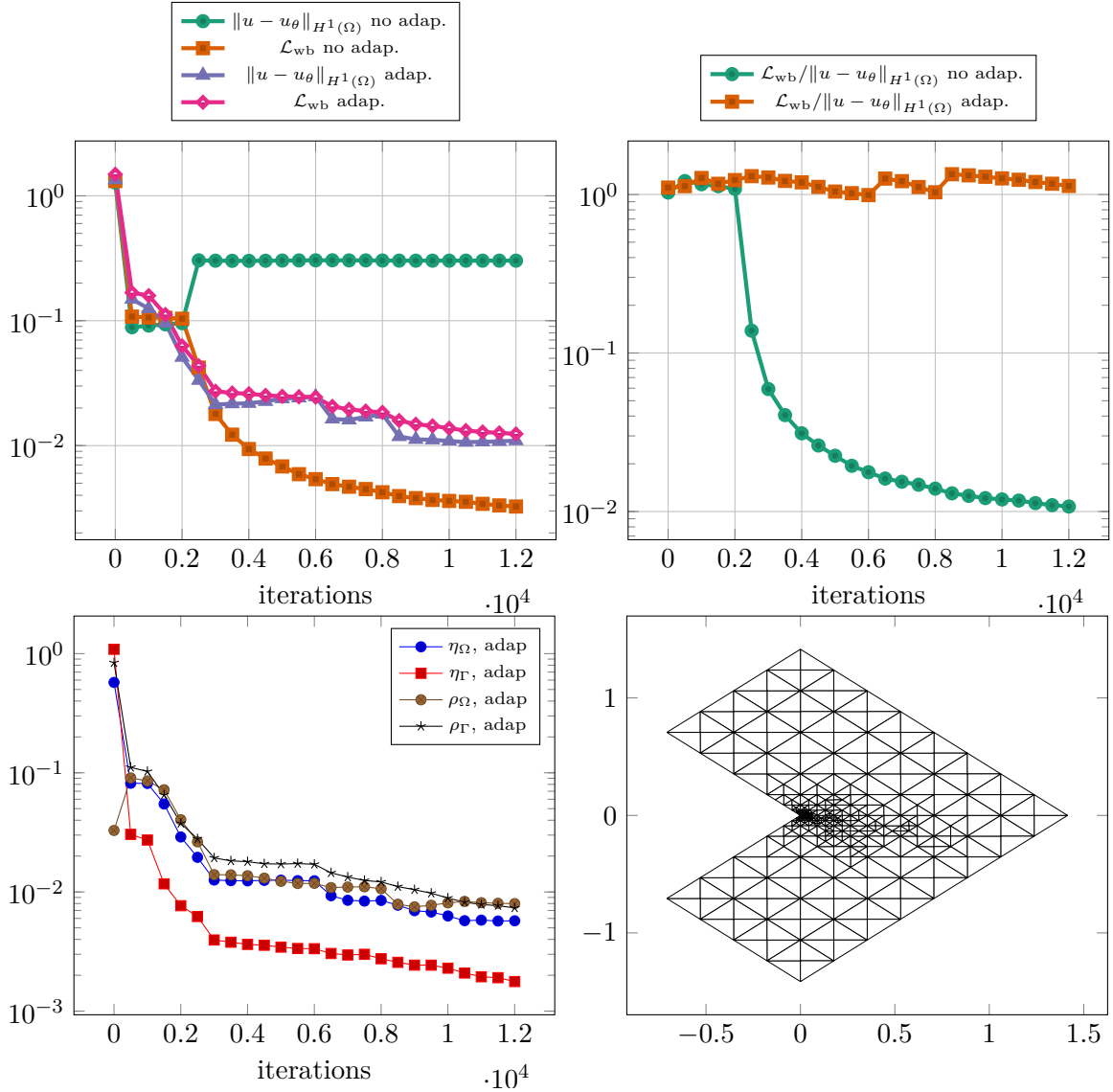

\subsection{Example for a convection-diffusion problem}\label{sec:ex:conv}
Let $\Omega = (0,1)^2$, $\bA = \varepsilon \boldsymbol{I}$ with $\boldsymbol{I}$ being the $2\times 2$-identity matrix, $\varepsilon = 0.05$ and $\bbeta = e^x(\sin(y),\cos(y))^\top$. 
We consider the manufactured solution
\begin{align*}
  u(x,y) = \tanh\left(\frac{1-\sqrt{x^2+y^2}}{\varepsilon}\right), \quad(x,y)\in \Omega,
\end{align*}
and set $g = u|_\Gamma$, $f = -\varepsilon\Delta u + \bbeta\cdot\nabla u$.
As mesh $\mesh$ we use a partition of $\Omega$ into $256$ triangles of equal area and $h_\mesh|_T = \tfrac18$ for all $T\in\mesh$.
We use a net with $L=2$ layers and $N=50$ neurons per layer, resulting in $2751$ learnable parameters. 
The net is trained using the losses $\loss_\mathrm{wb}$ and $\loss_{\mathrm{wb},\eta}$.
On each triangle, we use a $6$-point quadrature rule and a $4$-point Gaussian quadrature rule on boundary elements. 
Contrary to the prior examples, we consider the $\varepsilon$ dependent norms
\begin{align*}
  \|v\|_V := \sqrt{\varepsilon\|\nabla v\|_{\Omega}^2 + \|v\|_\Omega^2}, \quad\forall v\in V_\Omega = H_0^1(\Omega)
\end{align*}
to compute the loss functionals.

Figure~\ref{fig:conv} shows the results when using either $\loss_{\mathrm{wb},\eta}$ or $\loss_\mathrm{wb}$ as loss functional for the neural network training. As is the case for the problem from Section~\ref{sec:ex:etaVsEtaRho} we observe that using both contributions $\eta(u_\theta)$ and $\rho(u_\theta)$ as loss yields significantly smaller errors in the norm $\|\cdot\|_U := \sqrt{\varepsilon\|\nabla(\cdot)\|_\Omega^2 + \|\cdot\|_\Omega^2}$ than using $\eta(u_\theta)$ only.
A snapshot of the approximation $u_\theta$ after $5000$ iterations is presented in Figure~\ref{fig:conv} as well. 

\begin{figure}
  \begin{tikzpicture}
  \begin{groupplot}[group style={group size= 1 by 1},width=0.5\textwidth,cycle list/Dark2-6,
      cycle multiindex* list={
        mark list*\nextlist
      Dark2-6\nextlist},
      every axis plot/.append style={ultra thick},
      grid=major,
    ]
    \nextgroupplot[ymode=log,%xmode=log,
      legend entries={\tiny{$\|u-u_\theta\|_U$ using $\loss_{\mathrm{wb},\eta}$},\tiny{$\sqrt{\loss_{\mathrm{wb},\eta}}$},\tiny{$\|u-u_\theta\|_U$ using $\loss_{\mathrm{wb}}$},\tiny{$\sqrt{\loss_{\mathrm{wb}}}$}},
    legend pos=north east,xlabel={iterations}]
    \addplot table [x=iterations,y=errH1] {data/ConvectionEta.dat};
    \addplot table [x=iterations,y=errEst] {data/ConvectionEta.dat};
    \addplot table [x=iterations,y=errH1] {data/ConvectionEtaRho.dat};
    \addplot table [x=iterations,y=errEst] {data/ConvectionEtaRho.dat};
  \end{groupplot}
\end{tikzpicture}
  \raisebox{0.9cm}{\includegraphics[width=0.46\textwidth]{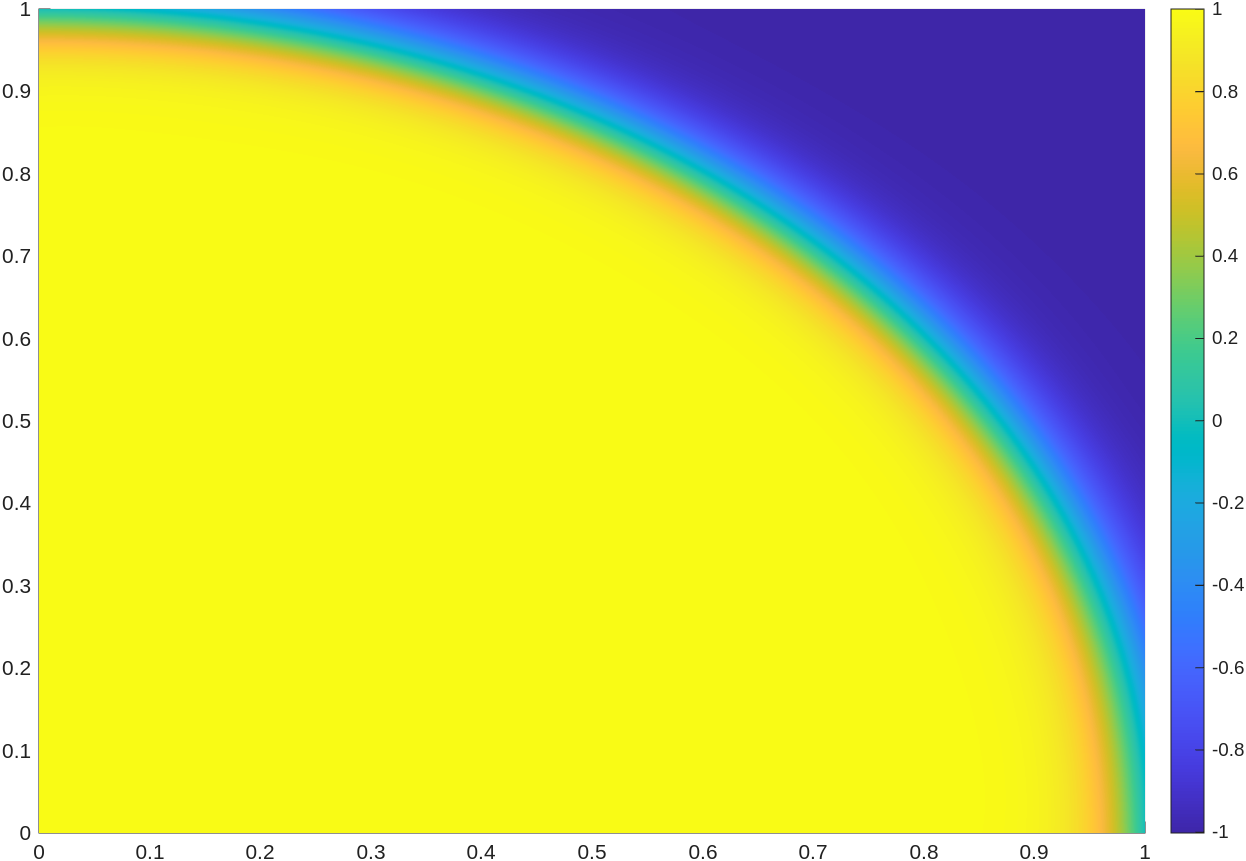}}
  \caption{Results for convection-diffusion problem from Section~\ref{sec:ex:conv}.
  The left plot compares losses to errors over iterations and the right plot visualizes the approximation $u_\theta$ at iteration $5000$ when using $\loss_{\mathrm{wb}}$ as loss functional.}\label{fig:conv}
\end{figure}

\section{Conclusions \& future work}\label{sec:conclusions}
We derived a general error estimate of the form 
\begin{align*}
  \|u-u_\theta\|_U \eqsim \eta(u_\theta) + \mu(u_\theta)
\end{align*}
where $\eta(u_\theta)$ is a computable discrete residual.
For several different formulations of a second-order elliptic model problem, we have estimated $\mu(u_\theta)$ by a computable upper bound $\rho(u_\theta)$ and proposed to use $\eta^2+\rho^2$ as (part of the) loss functional. 
In particular, we discussed the inclusion of Dirichlet boundary conditions and their relation to PINNs. 
Our numerical experiments suggest that the proposed use of $\eta^2+\rho^2$ as a loss functional in neural network training delivers not only a more robust loss but also generates more accurate approximations for the same number of iterations of the nonlinear solver. 
Furthermore, the importance of including boundary conditions in the loss functional rather than enforcing them is evident for problems with boundary layers.

In this work, we assumed that the neural network output is sufficiently smooth so that the residual estimators are well defined. This excludes the use of the popular ReLU activation function. In future work, we want to explore other variational formulations that allow less regular activation functions. 
Furthermore, we plan to study and refine the proposed adaptive algorithm and investigate singularly perturbed problems such as reaction-dominated and convection-dominated diffusion problems.

%Several aspects have intentionally been left outside the scope of the present work, which is already extensive and 
We primarily focused on establishing a general error control framework for neural network training. In particular, we did not account for quadrature errors in the evaluation of the loss functional. While our theoretical analysis neglects this effect, we demonstrate numerically, through a simple adaptive algorithm, that quadrature errors can significantly influence the training process, indicating that rigorous analysis of this effect is an important direction for future work. Moreover, we did not investigate the impact of neural network architecture choices or the systematic selection of discrete test spaces required to achieve prescribed convergence rates. Addressing these aspects, while crucial for practical implementations, goes beyond the scope of the present contribution and will be the subject of future studies.

\subsection*{Data/code availability}
All data generated for this work is included in the manuscript. 
The code can be made available on request. 

\section*{Declarations}

\subsection*{Conflicts of interest}
This work does not have any conflicts of interest.

%===================================================================================================
\bibliographystyle{alpha}
\bibliography{literature}

\newcommand{\etalchar}[1]{$^{#1}$}
\begin{thebibliography}{RMMnM{\etalchar{+}}24}

\bibitem[AD25]{AinsworthDong25}
Mark Ainsworth and Justin Dong.
\newblock Extended {G}alerkin neural network approximation of singular
  variational problems with error control.
\newblock {\em SIAM J. Sci. Comput.}, 47(3):C738--C768, 2025.

\bibitem[AFK{\etalchar{+}}13]{AFKPP13}
M.~Aurada, M.~Feischl, J.~Kemetm\"uller, M.~Page, and D.~Praetorius.
\newblock Each {$H^{1/2}$}-stable projection yields convergence and
  quasi-optimality of adaptive {FEM} with inhomogeneous {D}irichlet data in
  {$\Bbb R^d$}.
\newblock {\em ESAIM Math. Model. Numer. Anal.}, 47(4):1207--1235, 2013.

\bibitem[AGS16]{MR3522965}
Mark Ainsworth, Johnny Guzm\'an, and Francisco-Javier Sayas.
\newblock Discrete extension operators for mixed finite element spaces on
  locally refined meshes.
\newblock {\em Math. Comp.}, 85(302):2639--2650, 2016.

\bibitem[BCP22]{berrone2022solving}
Stefano Berrone, Claudio Canuto, and Moreno Pintore.
\newblock Solving pdes by variational physics-informed neural networks: an a
  posteriori error analysis.
\newblock {\em ANNALI DELL'UNIVERSITA'DI FERRARA}, 68(2):575--595, 2022.

\bibitem[BCPS23]{berrone2023enforcing}
Stefano Berrone, Claudio Canuto, Moreno Pintore, and Natarajan Sukumar.
\newblock Enforcing dirichlet boundary conditions in physics-informed neural
  networks and variational physics-informed neural networks.
\newblock {\em Heliyon}, 9(8), 2023.

\bibitem[BG09]{BG09}
Pavel~B. Bochev and Max~D. Gunzburger.
\newblock {\em Least-squares finite element methods}, volume 166 of {\em
  Applied Mathematical Sciences}.
\newblock Springer, New York, 2009.

\bibitem[BLM24]{badia2024finite}
Santiago Badia, Wei Li, and Alberto~F Mart{\'\i}n.
\newblock Finite element interpolated neural networks for solving forward and
  inverse problems.
\newblock {\em Computer Methods in Applied Mechanics and Engineering},
  418:116505, 2024.

\bibitem[BLP97]{MR1415797}
James~H. Bramble, Raytcho~D. Lazarov, and Joseph~E. Pasciak.
\newblock A least-squares approach based on a discrete minus one inner product
  for first order systems.
\newblock {\em Math. Comp.}, 66(219):935--955, 1997.

\bibitem[BLP98]{MR1602783}
James~H. Bramble, Raytcho~D. Lazarov, and Joseph~E. Pasciak.
\newblock Least-squares for second-order elliptic problems.
\newblock {\em Comput. Methods Appl. Mech. Engrg.}, 152(1-2):195--210, 1998.
\newblock Symposium on Advances in Computational Mechanics, Vol.\ 5 (Austin,
  TX, 1997).

\bibitem[BYZZ20]{WAN}
Gang Bao, Xiaojing Ye, Yaohua Zang, and Haomin Zhou.
\newblock Numerical solution of inverse problems by weak adversarial networks.
\newblock {\em Inverse Problems}, 36(11):115003, 31, 2020.

\bibitem[CCLL20]{cai2020deep}
Zhiqiang Cai, Jingshuang Chen, Min Liu, and Xinyu Liu.
\newblock Deep least-squares methods: An unsupervised learning-based numerical
  method for solving elliptic pdes.
\newblock {\em Journal of Computational Physics}, 420:109707, 2020.

\bibitem[CDG14]{DPGaposteriori}
Carsten Carstensen, Leszek Demkowicz, and Jay Gopalakrishnan.
\newblock A posteriori error control for {DPG} methods.
\newblock {\em SIAM J. Numer. Anal.}, 52(3):1335--1353, 2014.

\bibitem[CDG16]{BreakingSpaces16}
C.~Carstensen, L.~Demkowicz, and J.~Gopalakrishnan.
\newblock Breaking spaces and forms for the {DPG} method and applications
  including {M}axwell equations.
\newblock {\em Comput. Math. Appl.}, 72(3):494--522, 2016.

\bibitem[DG13]{PrimalDPG13}
L.~Demkowicz and J.~Gopalakrishnan.
\newblock A primal {DPG} method without a first-order reformulation.
\newblock {\em Comput. Math. Appl.}, 66(6):1058--1064, 2013.

\bibitem[DG25]{DPG_ActaNumerica}
Leszek Demkowicz and Jay Gopalakrishnan.
\newblock The discontinuous petrov–galerkin method.
\newblock {\em Acta Numerica}, 34:293–384, 2025.

\bibitem[ERU25]{ERU25}
Lewin Ernst, Nikolaos Rekatsinas, and Karsten Urban.
\newblock A posteriori certification for physics-informed neural networks.
\newblock {\em arXiv preprint arXiv:2502.20336}, 2025.

\bibitem[F{\"u}h21]{MultilevelNorms}
Thomas F{\"u}hrer.
\newblock Multilevel decompositions and norms for negative order {S}obolev
  spaces.
\newblock {\em Math. Comp.}, 91(333):183--218, 2021.

\bibitem[Hip97]{Hiptmair97}
R.~Hiptmair.
\newblock Multigrid method for {$\bold H({\rm div})$} in three dimensions.
\newblock volume~6, pages 133--152. 1997.
\newblock Special issue on multilevel methods (Copper Mountain, CO, 1997).

\bibitem[KGZ{\etalchar{+}}21]{krishnapriyan2021characterizing}
Aditi Krishnapriyan, Amir Gholami, Shandian Zhe, Robert Kirby, and Michael~W
  Mahoney.
\newblock Characterizing possible failure modes in physics-informed neural
  networks.
\newblock {\em Advances in Neural Information Processing Systems},
  34:26548--26560, 2021.

\bibitem[KZK19]{kharazmi2019variational}
Ehsan Kharazmi, Zhongqiang Zhang, and George~Em Karniadakis.
\newblock Variational physics-informed neural networks for solving partial
  differential equations.
\newblock {\em arXiv preprint arXiv:1912.00873}, 2019.

\bibitem[KZK21]{kharazmi2021hp}
Ehsan Kharazmi, Zhongqiang Zhang, and George~Em Karniadakis.
\newblock hp-vpinns: Variational physics-informed neural networks with domain
  decomposition.
\newblock {\em Computer Methods in Applied Mechanics and Engineering},
  374:113547, 2021.

\bibitem[LBH15]{lecun2015deep}
Yann LeCun, Yoshua Bengio, and Geoffrey Hinton.
\newblock Deep learning.
\newblock {\em nature}, 521(7553):436--444, 2015.

\bibitem[MSS24a]{MSS_QOLS24}
Harald Monsuur, Robin Smeets, and Rob Stevenson.
\newblock Quasi-optimal least squares: Inhomogeneous boundary conditions, and
  application with machine learning.
\newblock {\em arXiv preprint arXiv:2412.059965}, 2024.

\bibitem[MSS24b]{MSS24}
Harald Monsuur, Rob Stevenson, and Johannes Storn.
\newblock Minimal residual methods in negative or fractional {S}obolev norms.
\newblock {\em Math. Comp.}, 93(347):1027--1052, 2024.

\bibitem[Osw94]{Oswald94}
Peter Oswald.
\newblock {\em Multilevel finite element approximation}.
\newblock Teubner Skripten zur Numerik. [Teubner Scripts on Numerical
  Mathematics]. B. G. Teubner, Stuttgart, 1994.
\newblock Theory and applications.

\bibitem[RMMnM{\etalchar{+}}24]{RVPINN24}
Sergio Rojas, Pawe\l{} Maczuga, Judit Mu\~noz Matute, David Pardo, and Maciej
  Paszy\'nski.
\newblock Robust variational physics-informed neural networks.
\newblock {\em Comput. Methods Appl. Mech. Engrg.}, 425:Paper No. 116904, 18,
  2024.

\bibitem[RPK19]{raissi2019physics}
Maziar Raissi, Paris Perdikaris, and George~E Karniadakis.
\newblock Physics-informed neural networks: A deep learning framework for
  solving forward and inverse problems involving nonlinear partial differential
  equations.
\newblock {\em Journal of Computational physics}, 378:686--707, 2019.

\bibitem[ST21]{STBook21}
Ernst~P. Stephan and Thanh Tran.
\newblock {\em Schwarz methods and multilevel preconditioners for boundary
  element methods}.
\newblock Springer, Cham, [2021] \copyright 2021.

\bibitem[SvV20a]{MR4044445}
Rob Stevenson and Raymond van Veneti\"e.
\newblock Uniform preconditioners for problems of negative order.
\newblock {\em Math. Comp.}, 89(322):645--674, 2020.

\bibitem[SvV20b]{MR4094780}
Rob Stevenson and Raymond van Veneti\"e.
\newblock Uniform preconditioners for problems of positive order.
\newblock {\em Comput. Math. Appl.}, 79(12):3516--3530, 2020.

\bibitem[SZ90]{ScottZhang90}
L.~Ridgway Scott and Shangyou Zhang.
\newblock Finite element interpolation of nonsmooth functions satisfying
  boundary conditions.
\newblock {\em Math. Comp.}, 54(190):483--493, 1990.

\bibitem[TPM23]{taylor2023deep}
Jamie~M Taylor, David Pardo, and Ignacio Muga.
\newblock A deep {F}ourier residual method for solving {PDE}s using neural
  networks.
\newblock {\em Computer Methods in Applied Mechanics and Engineering},
  405:115850, 2023.

\bibitem[UBP{\etalchar{+}}25]{uriarte2025optimizing}
Carlos Uriarte, Manuela Bastidas, David Pardo, Jamie~M Taylor, and Sergio
  Rojas.
\newblock Optimizing variational physics-informed neural networks using least
  squares.
\newblock {\em Computers \& Mathematics with Applications}, 185:76--93, 2025.

\bibitem[UBR{\etalchar{+}}25]{udomworarat2025neural}
Tanakorn Udomworarat, Ignacio Brevis, Martin Richter, Sergio Rojas, and
  Kristoffer~G van~der Zee.
\newblock Neural network methods for power series problems of perron-frobenius
  operators.
\newblock {\em arXiv preprint arXiv:2505.05407}, 2025.

\bibitem[Ver94]{Verfuerth94}
R.~Verf\"urth.
\newblock A posteriori error estimation and adaptive mesh-refinement
  techniques.
\newblock In {\em Proceedings of the {F}ifth {I}nternational {C}ongress on
  {C}omputational and {A}pplied {M}athematics ({L}euven, 1992)}, volume~50,
  pages 67--83, 1994.

\bibitem[WYP22]{wang2022and}
Sifan Wang, Xinling Yu, and Paris Perdikaris.
\newblock When and why pinns fail to train: A neural tangent kernel
  perspective.
\newblock {\em Journal of Computational Physics}, 449:110768, 2022.

\bibitem[Y{\etalchar{+}}18]{yu2018deep}
Bing Yu et~al.
\newblock The deep ritz method: a deep learning-based numerical algorithm for
  solving variational problems.
\newblock {\em Communications in Mathematics and Statistics}, 6(1):1--12, 2018.

\end{thebibliography}
%===================================================================================================

\end{document}